\documentclass[11pt,reqno]{amsart}
\usepackage{amsmath,amssymb}
\usepackage[leqno]{amsmath}
\usepackage{agor}
\usepackage{color}
\usepackage{rotating}
\usepackage{lscape}
\usepackage{tablefootnote}
\usepackage[square,numbers,sort]{natbib}
\usepackage{todonotes}
\newcommand{\MIQO}{\text{MIQO}}
\newcommand{\NP}{\ensuremath{\mathcal NP}}
\newcommand{\revised}[1]{#1}

\newcommand{\leqnomode}{\tagsleft@true\let\veqno\@@leqno}
\newcommand{\reqnomode}{\tagsleft@false\let\veqno\@@eqno}

\def\SingleSpacedXI{\linespread{1.05}}
\SingleSpacedXI

\title[Convex hull of quadratics with indicators]{On the convex hull of convex quadratic optimization problems with indicators}

	\author{Linchuan Wei, Alper Atamt\"urk, Andr\'es G\'omez, Simge K\"u\c{c}\"ukyavuz}

\thanks{ \noindent \hskip -5mm
L. Wei: {Department of Industrial Engineering and Management Sciences, Northwestern University, \texttt{linchuanwei2022@u.northwestern.edu}.}\\
A. Atamt\"urk: {Department of Industrial Engineering and Operations Research, University of California Berkeley, \texttt{atamturk@berkeley.edu}.} \\
A. G\'omez: {Department of Industrial and System Engineering, University of Southern California, \texttt{gomezand@usc.edu}.}, \\
S. K\"u\c{c}\"ukyavuz: {Department of Industrial Engineering and Management Sciences, Northwestern University, \texttt{simge@northwestern.edu}.}\\
}
\RequirePackage[normalem]{ulem} 
\RequirePackage{color}\definecolor{RED}{rgb}{1,0,0}\definecolor{BLUE}{rgb}{0,0,1} 
\providecommand{\DIFaddbegin}{} 
\providecommand{\DIFaddend}{} 
\providecommand{\DIFdelbegin}{} 
\providecommand{\DIFdelend}{} 

\begin{document}

\begin{abstract} 
	We consider the convex quadratic optimization problem with indicator variables and arbitrary constraints on the indicators. We show that a convex hull description of the associated mixed-integer set in an extended space with a quadratic number of additional variables consists of a single positive semidefinite constraint (explicitly stated) and linear constraints. In particular, convexification of this class of problems reduces to describing a polyhedral set in an extended formulation. \revised{While the vertex representation of this polyhedral set is exponential and an explicit linear inequality description may not be readily available in general, we derive a compact mixed-integer linear formulation whose solutions coincide with the vertices of the polyhedral set.} We also give descriptions in the original space of variables: we provide a description based on an infinite number of conic-quadratic inequalities, which are ``finitely generated." In particular, it is possible to characterize whether a given inequality is necessary to describe the convex hull. The new theory presented here 
	unifies several previously established results, and paves the way toward utilizing polyhedral methods to analyze the convex hull of mixed-integer nonlinear sets.
\end{abstract}

\maketitle

\begin{center}
	{December 2021; September 2022} \\
\end{center}

\section{Introduction}
Given a symmetric positive semidefinite matrix $Q\in \R^{n\times n}$, vectors $a,b\in \R^n$ and set $Z\subseteq \{0,1\}^n$, consider the mixed-integer quadratic optimization (\MIQO) problem with indicator variables
\begin{subequations}\label{eq:miqo}
\begin{align}
\min\;&a^\top x+b^\top z+\tfrac{1}{2}t\\
(\MIQO)\qquad\text{s.t.}\;&x^\top Q x\leq t\label{eq:miqo_quad}\\
&x_i(1-z_i)=0, \ i=1,\dots,n\\
&x\in \R^n,\; z\in Z,\; t\in \R,
\end{align}
\end{subequations}
and the associated mixed-integer nonlinear set 
$$X=\left\{(x,z,t)\in \R^n\times Z\times \R: t\geq x^\top Qx,\; x\circ (\ones-z)=0\right\},$$
where $\ones$ denotes a vector of ones, and $x\circ (\ones-z)$ is the Hadamard product of vectors $x$ and $\ones-z$. There has recently been an increasing interest in problem \eqref{eq:miqo} due to its statistical applications: the nonlinear term \eqref{eq:miqo_quad} is used to model a quadratic loss function, as in regression, while $Z$ represents logical conditions on the support of the variables $x$. For example, given model matrix $F \in \R^{m\times n}$ 
 and responses $\beta\in \R^m$, setting $a=-\beta^\top F$, $Q=F^\top F$, $b=0$ and $Z=\left\{z\in \{0,1\}^n:\sum_{i=1}^nz_i\leq r\right\}$ \revised{in \eqref{eq:miqo} is equivalent to} the best subset selection problem with a given cardinality $r$ \cite{bertsimas2015or,cozad2014learning}:
 \revised{
 \begin{equation}\label{eqn:bestsubset}
   \min_{x, z} \; \|\beta - F x\|_2^{2} \quad \text{s.t.} \quad x\circ (\ones-z)=0, \; \sum_{i=1}^{n} z_i \leq r.   
 \end{equation}
}
Other constraints defining $Z$ that have been considered in statistical learning applications include multicollinearity  \cite{bertsimas2015or}, cycle prevention \cite{MKS21, KSMW20}, and hierarchy \cite{bien2013lasso}.
Set $X$ arises as a substructure in many other applications, including portfolio optimization \cite{B:miqp}, 
 optimal control \cite{Gao2011}, image segmentation  \cite{hochbaum2001efficient}, signal denoising \cite{bach2019submodular}.

A critical step toward solving \MIQO\ effectively is to convexify the set $X$.
Indeed, the mixed-integer optimization problem \eqref{eq:miqo} is equivalent to the convex optimization problem 
$$\min_{x,z,t} \bigg \{ a^\top x+b^\top z+\tfrac{1}{2}t \ : \ (x,z,t)\in \clconv(X) \bigg \},$$
where $\conv(X)$ denotes the convex hull of $X$ and $\clconv(X)$ is the closure of $\conv(X)$. 
However, problem \MIQO\  
is  \NP-hard even if $Z=\{0,1\}^n$ \cite{chen2014complexity}.  Thus, a simple description of $\clconv(X)$ is, in general, 
not possible unless \NP = Co-\NP. 

In practice, one aims to obtain a good convex relaxation of $X$, which can then be used  either as a standalone method (as is pervasively done in the machine learning literature),  to obtain high quality solutions via rounding, or  in a branch-and-bound framework. Nonetheless, it is unclear how to determine whether a given relaxation is good or not. In mixed-integer \emph{linear} optimization, it is well-understood that facet-defining inequalities give strong relaxations. However, in \MIQO\ (and, more generally, in mixed-integer nonlinear optimization problems), $\clconv(X)$ is not a polyhedron and there is no consensus on how to design good convex relaxations, or even what a good relaxation should be. 

An important class of convex relaxations of $X$ that has received attention in the literature is obtained by decomposing matrix $Q=\sum_{i=1}^\ell \Gamma_i+R$, where $\Gamma_i\succeq 0$, $i=1,\dots,\ell$, are assumed to be ``simple" and $R\succeq 0$. Then
\begin{equation}\label{eq:decomp}
t\geq x^\top Qx\Longleftrightarrow \revised{t \ge }\sum_{i=1}^\ell \tau_i+x^\top Rx\text{, and } \tau_i\geq x^\top \Gamma_ix,\; \forall i\in \{1,\dots,\ell\}, 
\end{equation} 
and each constraint $\tau_i\geq x^\top \Gamma_i x$ is replaced with a system of inequalities describing the convex hull of the associated ``simple" mixed-integer set. This idea was originally used in \cite{Frangioni2007}, where $\ell=n$, $(\Gamma_i)_{ii}=d_i>0$ and $(\Gamma_i)_{jk}=0$ otherwise, and constraints $\tau_i\geq d_i x_i^2$ are strengthened using the perspective relaxation \cite{Frangioni2006,akturk2009strong,Gunluk2010}, i.e., reformulated as $z_i\tau_i\geq d_i x_i^2$. Similar relaxations based on separable quadratic terms were considered in \cite{dong2015regularization,zheng2014improving}. A generalization of the above approach is \revised{rank-one decomposition}, which lets $\Gamma_i=h_ih_i^\top$ be a rank-one matrix \cite{atamturk2019rank,atamturk2020supermodularity,wei2020convexification,wei2021ideal}; in this case, letting $S_i=\left\{i\in [n]:h_i\neq 0\right\}$, constraints $\left(\sum_{j\in S_i}z_j\right)\tau_i\geq  (h_i^\top x)^2$ can be added to the formulation. 
Alternative generalizations of perspective relaxation that have been considered in the literature 
include exploiting substructures based on  $\Gamma_i$ where non-zeros are $2\times 2$ matrices
 \cite{Jeon2017,anstreicher2021quadratic,atamturk2018strong,atamturk2021sparse,frangioni2018decompositions,hga:2x2} or tridiagonal \cite{liu2021graph}.  

Convexifications based on decomposition~\eqref{eq:decomp} have proven to be strong computationally, and are attractive from a theoretical perspective. The fact that a given \revised{formulation} is ideal for the substructure $\tau_i \ge x^\top \Gamma_ix$ lends some theoretical weight to the strength of the convexification. However, approaches based on decomposition \eqref{eq:decomp} have fundamental limitations as well. First, they require computing the convex hull description of a nonlinear mixed-integer set to establish (theoretically) the strength of the relaxation, a highly non-trivial task that restricts the classes of matrices $\Gamma_i$ that can be used. Second, even if \revised{the ideal formulation for} the substructure $\tau_i \ge x^\top \Gamma_i x$ \revised{is available}, \revised{the convexification based on such decomposition} can still be a poor relaxation of $X$---and there is currently no approach to establish the strength of the relaxation without numerical computations. Third, it is unclear whether the structure of the relaxations induced by \eqref{eq:decomp} matches the structure of $\clconv(X)$, or if they are overly simple or complex.

\subsubsection*{Contributions and outline} In this paper, we close the aforementioned gaps in the literature by characterizing the structure of $\clconv(X)$. First, in \S\ref{sec:preliminaries}, we review relevant background for the paper. In \S\ref{sec:extended}, we show that $\clconv(X)$ can be described in a compact extended formulation with $\mathcal{O}(n^2)$ additional variables with linear constraints and a single positive semidefiniteness constraint. In particular, convexification of $X$ in this extended formulation reduces to describing a \revised{\emph{base}} polytope. \revised{We use the vertex description of this base polytope, which is exponential in general. However, we show that the set of vertices can be represented as the feasible points of a compact mixed-integer linear formulation (\S\ref{sec:MILP})}.  In \S\ref{sec:original}, we characterize $\clconv(X)$ in the original space of variables. While the resulting description has an infinite number of conic quadratic constraints,
we show that $\clconv(X)$ is \emph{finitely generated}, and thus we establish which inequalities 
are necessary to describe $\clconv(X)$---in precisely the same manner that facet-defining inequalities are required to describe a polyhedron. We also establish a relationship between $\clconv(X)$ and relaxations obtained from decompositions \eqref{eq:decomp}. \revised{In \S\ref{sec:MILP}, we present a mixed-integer \textit{linear} formulation of the \MIQO\  problem using the theoretical results in \S\ref{sec:extended}}. Finally, in \S\ref{sec:conclusion} we conclude the paper with a few  remarks.

 \revised{
We point out that, using standard disjunctive programming techniques \cite{CS:dis-conv}, it is possible to obtain a conic quadratic extended formulation of \eqref{eq:miqo}, although such representation typically requires adding $\mathcal{O}(|Z|n)$ number of variables and $\mathcal{O}(|Z|)$ \emph{nonlinear} constraints. Since $|Z|$ is often exponential in $n$, these formulations are in general impractical, and therefore their use has been restricted to small instances with $n\leq 2$ \cite{Gunluk2010,atamturk2018strong,anstreicher2021quadratic,frangioni2018decompositions,hga:2x2} or problems with special structures that admit a compact representation \cite{han2021compact}. We argue that the convexifications in this paper are significantly more tractable: regardless of $Z$, we require only $\mathcal{O}(n^2)$ variables instead of $\mathcal{O}(|Z|n)$, and only \emph{one} nonlinear conic constraint instead of $\mathcal{O}(|Z|)$. The major complexity of the proposed formulations in this paper is the exponential number of \emph{linear} inequalities, which can be generated, as needed, using mature mixed-integer linear optimization techniques.
 }

\DIFdelend
\section{Notation and Preliminaries}\label{sec:preliminaries}

In this section, we first review the relevant background and introduce the notation used in the paper.

\begin{definition}[\cite{penrose1955generalized}]\label{def:pseudoinverse} Given a matrix $W\in \R^{p\times q}$, its pseudoinverse $W^\dagger\in \R^{q\times p}$ is the unique matrix satisfying the four properties:
\[ 	WW^\dagger W =W, \ \
W^\dagger WW^\dagger =W^\dagger,\ \
	(WW^\dagger)^\top =(WW^\dagger),\ \
	(W^\dagger W)^\top =W^\dagger W.
\]
\end{definition}
Clearly, if $W$ is invertible, then $W^{-1}=W^\dagger$. It also readily follows from the definition that $(W^\dagger)^\dagger=W$.

We recall the generalized Schur complement, relating pseudoinverses and positive semidefinite matrices. 
\begin{lemma}[\cite{albert1969conditions}]\label{lem:schur}
Let $W=\small \begin{pmatrix}W_{11}&W_{12}\\
W_{12}^\top & W_{22}\end{pmatrix}$, with symmetric $W_{11}\in \R^{p\times p}$, symmetric $W_{22}\in \R^{q\times q}$, and $W_{12}\in \R^{p\times q}$. Then $W\succeq 0$ if and only if $W_{11}\succeq 0$, $W_{11}W_{11}^\dagger W_{12}=W_{12}$ and $W_{22}-W_{12}^\top W_{11}^\dagger W_{12}\succeq 0$.  
\end{lemma}
Note that if $W_{11}\succ 0$, then the second condition of  Lemma \ref{lem:schur} is automatically satisfied. Otherwise, this condition is equivalent to the system of equalities $W_{11}U=W_{12}$ having a solution $U\in \R^{p\times q}$.

Let $[n]=\{1,\dots,n\}$. Throughout, we use the convention that $x_i^2/z_i=0$ if $x_i=z_i=0$ and $x_i^2/z_i=+\infty$ if $z_i=0$ and $x_i\neq 0, i\in[n]$.  Given two matrices $V,W$ of matching dimensions, let $\langle V,W\rangle=\sum_{i}\sum_jV_{ij}W_{ij}$ denote the usual inner product. Given a matrix $W\in \R^{n \times n}$, let $\text{Tr}(W)=\sum_{i=1}^nW_{ii}$ denote its trace, and let $W^{-1}$ denote its inverse, if it exists. \revised{$\|W\|_2$ and $\|W\|_{\infty}$ denote the \textit{Frobenius} norm and the maximum absolute value of entries of $W$ respectively, and $\lambda_{\max}(W)$ means the maximum eigenvalue of $W$. }We let $\mathrm{col}(W)$  denote the column  space of matrix $W$.
Given a matrix $W\in \R^{n\times n}$ and $S\subseteq [n]$, let $W_S\in \R^{S\times S}$ be the submatrix of $W$ induced by $S$, and let $\hat W_S\in \R^{n\times n}$ be the $n\times n$ matrix obtained from $W_S$ by filling the missing entries with zeros, i.e.,  matrices  subscripted by $S$ without ``hat" refer to the lower-dimensional submatrices. \revised{For any two sets $S, T \subset [n]$, let $W_{S,T}$ denote the submatrix of $W$ with rows in $S$ and columns in $T$.} Note that if matrix $W\succ 0$, then it can be easily be verified from Definition~\ref{def:pseudoinverse} that the submatrix of $\hat W_S^{\dagger}$ indexed by $S$ coincides with $W_S^{-1}$, and $\hat W_S^{\dagger}$ is zero elsewhere; in this case, we abuse notation and write $\hat W_S^{-1}$ instead of $\hat W_S^{\dagger}$.  Given $S\subseteq [n]$, let $\hat \ones_S\in \{0,1\}^n$  
be the indicator vector of $S$. We define $\pi_S$ as the projection onto the subspace indexed by $S$ and $\pi_{S}^{-1}(x)$ as the preimage of $x$ under $\pi_S$.

\begin{example}
	Let $Q=\begin{pmatrix} d_1& b\\
		b &d_2
	\end{pmatrix}$ with $d_1,d_2>0$ and $d_1d_2>b^2$. Then
	\begin{align*}
		&\hat Q_\emptyset^{-1}=\begin{pmatrix} 0& 0\\
			0 &0
		\end{pmatrix},\;\hat Q_{\{1\}}^{-1}=\begin{pmatrix} 1/d_1& 0\\
			0 &0
		\end{pmatrix},\;
		\hat Q_{\{2\}}^{-1}=\begin{pmatrix} 0& 0\\
			0 &1/d_2
		\end{pmatrix},\text{ and }\\&Q_{\{1,2\}}^{-1}=\frac{1}{d_1d_2-b^2}\begin{pmatrix} d_2& -b\\
			-b &d_1
		\end{pmatrix}.
	\end{align*}
\end{example}

\section{Convexification in an extended space}\label{sec:extended}

In this section, we describe $\clconv(X)$ in an extended space. In \S\ref{sec:canonical}, we provide a ``canonical" representation of $\clconv(X)$ under the assumption that $Q\succ 0$. In \S\ref{sec:separable}, we provide alternative representations of $\clconv(X)$, which can handle non-invertible matrices $Q$ and may \revised{also} lead to sparser formulations. 

\subsection{Canonical representation}\label{sec:canonical}

Given $Q\succ 0$, define the polytope $P\subseteq \R^{n+n^2}$ as 
$$P\defeq\conv\left(\left\{(\hat \ones_S,\hat Q_S^{-1})\right\}_{\revised{\hat \ones_{S} \in Z}}\right).$$
Proposition~\ref{prop:valid} below shows how to construct mixed-integer conic formulations of \MIQO\ using polytope $P$. 
\begin{proposition}\label{prop:valid}
	If $Q\succ 0$, then the mixed-integer optimization model
	\begin{subequations}\label{eq:miqoSDP}
		\begin{align}
		\min_{x,z,W,t}\;&a^\top x+b^\top z+\tfrac{1}{2}t\\
		\text{s.t.}\;&\begin{pmatrix}W & x\\
		x^\top&t\end{pmatrix}\succeq 0\label{eq:miqoSDP_psd}\\
		&(z,W)\in P\label{eq:miqoSDP_polytope}\\
		& z\in \{0,1\}^n\label{eq:miqoSDP_feasible}\\
		&x\in \R^n,  t\in \R  \label{eq:miqoSDP_x}
		\end{align}
	\end{subequations}
	is a valid formulation of problem \eqref{eq:miqo}. 
\end{proposition}
\begin{proof}
	Consider a point $(x,z,t, W)$ \revised{satisfying constraints \eqref{eq:miqoSDP_psd}, \eqref{eq:miqoSDP_polytope}} with $z=\hat \ones_S$ for some \revised{$\hat e_{S} \in Z$}. Constraint \eqref{eq:miqoSDP_polytope} is satisfied if and only if $W=\hat Q_S^{-1}$. Therefore, constraint \eqref{eq:miqoSDP_psd} reduces to 
	$$\small \begin{pmatrix}
	Q_S^{-1} &  \mathbf{0} &  x_S\\
	\mathbf{0} & \mathbf{0} & x_{[n]\setminus S}\\
	x_{S}^\top & x_{[n]\setminus S}^\top & t
	\end{pmatrix}\succeq 0.$$
	Since the pseudoinverse of matrix $W=\begin{pmatrix} 	Q_S^{-1} &  \mathbf{0}\\
	\mathbf{0} & \mathbf{0}\end{pmatrix}$ is $W^\dagger=\begin{pmatrix} 	Q_S &  \mathbf{0}\\
	\mathbf{0} & \mathbf{0}\end{pmatrix}$, we find 
	from Lemma~\ref{lem:schur} that constraint \eqref{eq:miqoSDP_psd} is satisfied if and only if: 
	\begin{itemize}
\item  $W\succeq 0$, which is automatically satisfied.  
	
\item 	$ WW^\dagger x=x\Leftrightarrow \begin{pmatrix} I &  \mathbf{0}\\
	\mathbf{0} & \mathbf{0}\end{pmatrix}\begin{pmatrix}x_S\\x_{[n]\setminus S}\end{pmatrix}=\begin{pmatrix}x_S\\x_{[n]\setminus S}\end{pmatrix}\Leftrightarrow x_{[n]\setminus S}=0.$ 
	Thus, condition $WW^\dagger x=x$ simply enforces the complementarity constraints $x\circ(\ones-z)=0$. 
	
	\item $t\geq x^\top W^\dagger x\Leftrightarrow t\geq x_S^\top Q_Sx_S$, which is precisely the nonlinear constraint defining set $X$. 
	\end{itemize}
\revised{Now, it is clear that for any $(x, z, t, W)$ satisfying constraints \eqref{eq:miqoSDP_psd}, \eqref{eq:miqoSDP_polytope}, \eqref{eq:miqoSDP_feasible}, it holds $(x, z, t) \in X$. On the other hand, for any $(x, z, t) \in X$ with $z = \hat e_{S}$ for some $S \subset [n]$, we can always let $W = \hat Q_{S}^{-1}$ and similarly, $(x, z, W, t)$ satisfies constraints \eqref{eq:miqoSDP_psd}, \eqref{eq:miqoSDP_polytope}, \eqref{eq:miqoSDP_feasible}.}
\end{proof}

Note that condition $WW^\dagger x=x$ is used to enforce the complementarity constraints. We point out that a similar idea was recently used in the context of low-rank optimization \cite{bertsimas2020mixed}.

Now consider the convex relaxation of \eqref{eq:miqoSDP}, obtained by dropping the integrality constraints $z\in \{0,1\}^n$: 
\begin{subequations}\label{eq:miqoRelax}
	\begin{align}
	\min_{x,z,W,t}\;&a^\top x+b^\top z+\tfrac{1}{2}t\\
	\text{s.t.}\;&\eqref{eq:miqoSDP_psd},\eqref{eq:miqoSDP_polytope},\eqref{eq:miqoSDP_x}.
	\end{align}
\end{subequations}

\revised{
\begin{theorem}\label{thm:canonical}
Let $Q$ be a positive definite matrix. Then
\begin{equation*}
\clconv(X) \; = \; \{(z, x, t) \in [0,1]^{n} \times \R^{n+1} \; | \; \exists W \in \R^{n \times n} \text{s.t.}    \; \eqref{eq:miqoSDP_psd},\eqref{eq:miqoSDP_polytope}\}.
\end{equation*}
Consequently, the problem \eqref{eq:miqoRelax} has an optimal solution integral in $z$.
\end{theorem}}

\begin{proof}
	\revised{
	First observe that constraints \eqref{eq:miqoSDP_psd},\eqref{eq:miqoSDP_polytope} define a closed convex set.}
	Projecting out variable $t$, we find that problem \eqref{eq:miqoRelax} reduces to 
	\begin{subequations}\label{eq:miqoRelax2-1}
		\begin{align}
		\min_{x,z,W}\;&a^\top x+b^\top z+\tfrac{1}{2}x^\top W^\dagger x\\
		\text{s.t.}\;&WW^\dagger x=x\label{eq:bounded}\\
		&(z,W)\in P,\; x\in \R^n.
		\end{align}
	\end{subequations}
\revised{Note that this formulation uses the pseudoinverse of a matrix of variables.} Observe that we omit the constraint $W\succeq 0$. Since every extreme point $(\bar z, \bar W)$ of $P$ satisfies $\bar W\succeq 0$, it follows $(z,W)\in P$ already implies $W\succeq 0$. 

We argue that for any fixed $(z, W)\in P$, setting $x=-Wa$ is optimal for \eqref{eq:miqoRelax2-1}. Using equality \eqref{eq:bounded}, we replace the term $a^\top x$ in the objective with $a^\top WW^\dagger x$. Since the problem is convex in $x$, from KKT conditions we find that any point $x$ satisfying
\begin{subequations}\label{eq:KKT}
\begin{align}
& WW^\dagger x=x\label{eq:KKT_primalfeas}\\
&\exists \lambda\in \R^n \text{ s.t. } W^\dagger  Wa+ W^\dagger x+\lambda^\top( W W^\dagger -I)=0\label{eq:KKT_dualfeas}
\end{align}
\end{subequations}
is optimal. In particular, setting $x=-Wa$, we find that \eqref{eq:KKT_dualfeas} is satisfied with $\lambda=0$, and \eqref{eq:KKT_primalfeas} is satisfied since $W W^\dagger x=- W W^\dagger  Wa=- Wa=x$.

 Substituting $x=-Wa$ in the relaxed problem, we obtain 
	\begin{subequations}\label{eq:miqoRelax3-1}
		\begin{align}
		\min_{z,W}\;&-\tfrac{1}{2} a^\top Wa+b^\top z\\
		\text{s.t.}\;&(z,W)\in P.
		\end{align}
	\end{subequations}
	Since the objective $-\frac{1}{2} \langle aa^\top,W\rangle+b^\top z$ is linear in $(z,W)$ and $P$ is a polytope, there exists an optimal solution $(z^*,W^*)$ that is an extreme point of $P$, and in particular there exists $\hat \ones_{S} \in Z$ such that $z^*=\hat\ones_S$ and $W^*=\hat Q_S^{-1}$.

\end{proof}

\begin{remark}
	The convexification for the case where $Q$ is tridiagonal \citep{liu2021graph} is precisely in the form given in Theorem~\ref{thm:canonical}, where the polyhedron $P$ is described with a compact extended formulation. \qed
\end{remark}

\subsubsection{Bivariate quadratic functions}\label{sec:2x2}
Consider set 
\begin{align*}
X_{2\times 2}\!=\!\left\{(x,z,t)\in \R^2\!\times \!\{0,1\}^n\!\times \!\R: t\geq d_1x_1^2-2x_1x_2+d_2x_2^2,\; x\!\circ \! (\ones-z)\!=\!0\right\},
\end{align*}
where $d_1d_2>1, d_1,d_2 > 0$. 
Set $X_{2\times 2}$ corresponds (after scaling) to a generic strictly convex quadratic function of two variables; conic quadratic disjunctive programming representations of $\clconv(X_{2\times 2})$ have been used in the literature \cite{anstreicher2021quadratic}, explicit representations of $\clconv\left(X_{2\times 2}\cap \{(x,z,t):x\geq 0\}\right)$ in the original space of variables have been given \cite{atamturk2021sparse,hga:2x2},  and descriptions of the rank-one case $d_1d_2=1$ were given in \cite{atamturk2019rank}. \revised{A description of $\clconv \left(X_{2 \times 2} \cap \{(x,z,t):\ell\leq x \leq u\}\right)$ in a conic quadratic extended formulation is given in \cite{frangioni2018decompositions} using disjunctive programming. This formulation can be easily adapted to the case with no bounds (considered here), and requires three additional variables and three conic quadratic constraints to use with solvers. }
We now give a more compact representation of $\clconv(X_{2\times 2})$ with free variables.

 We now illustrate Theorem~\ref{thm:canonical} by computing an extended formulation of $\clconv(X_{2\times 2})$, that is, for $Q=\small \begin{pmatrix}d_1&-1\\-1&d_2\end{pmatrix}$. Let $\Delta:=d_1d_2-1 >0$ be the determinant of $Q$. 

\begin{proposition}\label{prop:2x2extended}
The closure of the convex hull of $X_{2\times 2}$ is
	\begin{align*}
	\clconv(X_{2\times 2})=&\Bigg\{(x,z,t)\in \R^5:\exists W\in \R^{2\times 2}\text{ such that }
	\small \begin{pmatrix}W_{11}&W_{12}&x_1\\
	W_{12}&W_{22}&x_2\\
	x_1&x_2&t\end{pmatrix}\succeq 0,\\
	&0\leq z_1\leq 1,\; 0\leq z_2\leq 1,\;d_1W_{11}=W_{12}+z_1,\; d_2W_{22}=z_2+W_{12},\\
	&W_{12}\geq 0,\; \Delta W_{12}\geq -1+z_1+z_2,\;\Delta W_{12}\leq z_1,\; \Delta W_{12}\leq z_2 \Bigg\}.
	\end{align*}
\end{proposition}
\begin{proof}
	Polyhedron $P$ is the convex hull of the four points given in Table~\ref{tab:P2times2}. 
	\begin{table}[!h]
		\begin{center}
\caption{Extreme points of $P$ corresponding to set $X_{2\times 2}$.}
\label{tab:P2times2}
	\begin{tabular}{c c c} \hline \hline
		$z_1$&$z_2$&$W$\\
		\hline
		$0$&$0$&$\small \begin{pmatrix}0&0\\0&0\end{pmatrix}$\\
		$1$&$0$&$\small \begin{pmatrix}1/d_1&0\\0&0\end{pmatrix}$\\
		$0$&$1$&$\small \begin{pmatrix}0&0\\0&1/d_2\end{pmatrix}$\\
		$1$&$1$&$\small \frac{1}{\Delta}\begin{pmatrix}d_2&1\\1&d_1\end{pmatrix}$\\
		\hline \hline
	\end{tabular}
\end{center}
\end{table}

Note that equalities $W_{11}=\frac{1}{d_1}(z_1+W_{12})$ and  $W_{22}=\frac{1}{d_2}(z_2+W_{12})$ are valid. Letting $w=W_{12}$ and projecting out variables $W_{11}$ and $W_{22}$, we find that \begin{equation}\small 
\label{eq:matrixFormW}W=\begin{pmatrix}\frac{1}{d_1}z_1 &0\\0&\frac{1}{d_2}z_2\end{pmatrix}+\begin{pmatrix}1/d_1&1\\1&1/d_2\end{pmatrix}w.
\end{equation}
\revised{
Also note that $w = \frac{1}{\Delta} \min \{z_1, z_2\}$, and the convex hull of $\big \{(z_1, z_2, w) \in \{0,1\}^2 \times \R \; | \; w = \frac{1}{\Delta} \min \{z_1, z_2\} \big \}$ is described by the following inequalities:
\begin{align}\label{eq:2times2_diag} \small
&w\geq 0,\; w\geq \frac{1}{\Delta}(-1+z_1+z_2),\; w\leq \frac{1}{\Delta}z_2,\; w\leq \frac{1}{\Delta}z_1, 0 \leq z_1, z_2 \leq 1
\end{align}
Then, \eqref{eq:matrixFormW} and \eqref{eq:2times2_diag} describe the polyhedron $P$.
}

\end{proof}
\begin{remark}
	Since $P$ is not full-dimensional, we require only one additional variable $w$ (instead of three) for conic representation of $\clconv(X_{2\times 2})$ via the constraints $0\leq z\leq 1$, \eqref{eq:2times2_diag}, and
	$$\small \begin{pmatrix}(1/d_1)(z_1+w)&w&x_1\\
	w&(1/d_2)(z_2+w)&x_2\\
	x_1&x_2&t\end{pmatrix}\succeq 0.$$
	\qed
\end{remark}

\begin{remark}\label{rem:McCormick}
	The matrix representation \eqref{eq:matrixFormW} suggests an interesting connection between $\clconv(X_{2\times 2})$ and McCormick envelopes. Indeed, from Table~\ref{tab:P2times2}, we see that
	$$\small W=\begin{pmatrix}1/d_1&0\\0&0\end{pmatrix}z_1+\begin{pmatrix}0&0\\0&1/d_2\end{pmatrix}z_2+\frac{1}{\Delta}\begin{pmatrix}1/d_1&1\\1&1/d_2\end{pmatrix}z_1z_2.$$
	Moreover, the usual McCormick envelopes of the bilinear term $z_1z_2$, given by $\max\{0,-1+z_1+z_2\}\leq z_1z_2\leq \min\{z_1,z_2\}$, are sufficient to characterize the convex hull. 
	 \qed
\end{remark}

\subsubsection{Quadratic functions with ``choose-one" constraints}\label{sec:choose1}

Given $Q\succ0$, consider set 
\begin{align*}
X_{C1}\!=\!\left\{\!(x,z,t)\in \R^n\!\times \!\{0,1\}^n\!\times \R: t\geq x^\top Qx,\; x\circ (\ones-z)=0,\;\sum_{i=1}^nz_i\leq 1\right\} \cdot
\end{align*}
Set $X_{C1}$ arises, for example, in regression problems with multicollinearity constraints \cite{bertsimas2015or}: given a set of $J$ features that are collinear, constraints $\sum_{i\in J}z_i\leq 1$ are used to ensure that at most one such feature is chosen.

	The closure of the convex hull of $X_{C1}$ is \revised{\cite[see, e.g.,][]{frangioni2018decompositions,wei2020convexification}}
	\begin{align*}
	\clconv(X_{C1})=&\Bigg\{(x,z,t)\in \R^{n}\times \R_+^n\times \R:t\geq \sum_{i=1}^n Q_{ii}x_i^2/z_i,\; \sum_{i=1}^nz_i\leq 1 \Bigg\} \cdot
	\end{align*}

\revised{We now give an alternative derivation of this result using our technique.}	Polyhedron $P$ is the convex hull of $n+1$ points: point $(0,0)$ and points $\{(\hat\ones_{\{i\}},\hat Q_{\{i\}}^{-1})\}_{i=1}^n$. It can easily be seen that $P$ is described by constraints $W_{ij}=0$ whenever $i\neq j$, $W_{ii}=z_i/Q_{ii}$ for $i\in [n]$, and constraints $z\geq 0$, $\sum_{i=1}^n z_i\leq 1$. In particular, constraint \eqref{eq:miqoSDP_psd} reduces to 
	\begin{align*}\small 
 &\begin{pmatrix}z_1/Q_{11}&0&\dots&0&x_1\\
	0&z_2/Q_{22}&\dots&0&x_2\\
	\vdots&\vdots&\ddots&\vdots&\vdots\\
	0&0&0&z_n/Q_{nn}&x_n\\
	x_1&x_2&\dots&x_n&t
	\end{pmatrix}\succeq 0,\end{align*}
which \revised{by Lemma ~\eqref{lem:schur}} is equivalent to
\[
t\geq \sum_{i=1}^n Q_{ii}x_i^2/z_i,\; z_i/Q_{ii}\geq 0,
\]
\revised{and $x_i = 0$ if $z_i/Q_{ii}=0, \; \forall i \in [n]$. Note that the second condition is the complementarity constraint, which is already included in the constraint $t\geq \sum_{i=1}^n Q_{ii}x_i^2/z_i$. }
\subsection{Factorable representation}\label{sec:separable}
 
A (possibly low-rank) 
matrix $Q\in \R^{n\times n}$ is positive semidefinite if and only if there exists  some $F\in \R^{n\times k}$ such  that $Q=FF^\top$. Then, letting $u=F^\top x$, one can rewrite $x^\top Qx$ as $x^\top FF^\top x=u^\top u$. Matrix $F$ may be immediately available when formulating the problem, or may be obtained through a Cholesky decomposition or eigendecomposition of $Q$. Such a factorization is often employed by solvers, since it results in simpler (separable) nonlinear terms, and in many situations matrix $F$ is sparse as well. In this section, we discuss representations of $\clconv(X)$ amenable to such factorizations of $Q$.  
While the proofs of the propositions of this section are similar to those in Section \ref{sec:canonical},  additional care is required to handle unbounded problems~\eqref{eq:miqo} arising from a rank-deficient $Q$. 

Given $F\in \R^{n\times k}$, define $F_S\in \R^{S\times k}$ as the submatrix of $F$ corresponding to the rows indexed by $S$, and let $\hat F_S\in \R^{n\times k}$ 
be the matrix obtained by filling the missing entries with zeros. 
Define the polytope $P_F\subseteq \R^{n+k^2}$ as 
$$
P_F=\conv\left(\left\{(\hat \ones_S,\hat F_S^\dagger \hat F_S)\right\}_{\revised{\hat \ones_{S} \in Z}}\right) \cdot
$$

\begin{remark}\label{rem:projection}
For any $S\subseteq [n]$, matrix $\hat F_S^\dagger \hat F_S$ is an orthogonal projection matrix (symmetric and idempotent), and in particular $(\hat F_S^\dagger \hat F_S)^\dagger=\hat F_S^\dagger \hat F_S$. These properties can be easily verified from Definition~\ref{def:pseudoinverse}. Since all eigenvalues of an orthogonal projection matrix are either $0$ or $1$, it also follows that $\hat F_S^\dagger \hat F_S\succeq 0$.  \qed
\end{remark}

\begin{proposition}\label{prop:validSep}
	If $Q=FF^\top$, then the mixed-integer optimization model
	\begin{subequations}\label{eq:miqoSDPSep2}
		\begin{align}
		\min_{x,z,W,t}\;&a^\top x+b^\top z+\tfrac{1}{2}t\\
		\text{s.t.}\;&\begin{pmatrix}W & F^\top x\\
		x^\top F&t\end{pmatrix}\succeq 0\label{eq:miqoSDPSep2_psd}\\
		&(z,W)\in P_F\label{eq:miqoSDPSep2_polytope}\\
		& z\in \{0,1\}^n,\; x\circ (\ones-z)=0\label{eq:miqoSDPSep2_feasible}\\
		&x\in \R^n, t\in \R\label{eq:miqoSDPSep2_x}
		\end{align}
	\end{subequations}
	is a valid formulation of problem \eqref{eq:miqo}. 
\end{proposition}

\begin{proof}
	Consider a point $(x,z,t)\in X$ with $z=\hat \ones_S$ for some $\hat \ones_{S} \in Z$. Constraint \eqref{eq:miqoSDPSep2_feasible} is trivially satisfied. Constraint \eqref{eq:miqoSDPSep2_polytope} is satisfied if and only if $W=\hat F_S^{\dagger}\hat F_S$. Note that in any feasible solution, $x_i=0$ whenever $i\not\in S$, and in particular $F^\top x=\hat F_{S}^\top x$. From Lemma~\ref{lem:schur}, we find that constraint \eqref{eq:miqoSDPSep2_psd} is satisfied if and only if (recall properties in Remark~\ref{rem:projection}): 
	\begin{itemize}
		\item  $\hat F_S^{\dagger}\hat F_S \succeq 0$, which is automatically satisfied.  
		
		\item 	$ \hat F_S^{\dagger}\hat F_S(\hat F_S^{\dagger}\hat F_S)^\dagger F^\top x=F^\top x.$  
		  We find that $$\hat F_S^{\dagger}\hat F_S(\hat F_S^{\dagger}\hat F_S)^\dagger \hat F_S^\top x=\hat F_S^{\dagger}\hat F_S\hat F_S^{\dagger}\hat F_S \hat F_S^\top x=\hat F_S^{\dagger}\hat F_S \hat F_S^\top x=\hat F_S^{^\top}(\hat F_S^\dagger)^\top \hat F_S^\top x=\hat F_S^\top  x,$$
		 and, therefore, this condition is satisfied as well. 
		
		\item $t\geq x^\top F W^\dagger F^\top x\Leftrightarrow t\geq x_S^\top \hat F_S (\hat F_S^\dagger \hat F_S)^\dagger\hat F_S^\top x_S=x_S^\top \hat F_S \hat F_S^\dagger \hat F_S\hat F_S^\top x_S=x_S^\top \hat F_S\hat F_S^\top x_S$, which is precisely the nonlinear constraint defining set $X$ and is thus satisfied. 
	\end{itemize}
	
\end{proof}

While the proofs of Proposition~\ref{prop:valid} and \ref{prop:validSep} are similar in spirit, we highlight a critical difference. In the proof of Proposition~\ref{prop:valid}, with the assumption $Q\succ 0$, constraints $WW^\dagger x=x$ enforce the complementarity constraints $x\circ (\ones-z)=0$, and therefore, such constraints are excluded in \eqref{eq:miqoSDP}. In contrast, in the proof of Proposition~\ref{prop:validSep}, with $Q$ potentially of low-rank, 
constraints $WW^\dagger F^\top x=F^\top x$ alone are not sufficient to enforce $x\circ (\ones-z)=0$, and therefore, they are included in \eqref{eq:miqoSDPSep2} and are used to prove the validity of the mixed-integer formulation.
Indeed, if there exist $\hat \ones_{S} \in Z$ and $\bar x\in \R^{n}$ such that $\bar x_S\neq 0$, $\bar x_{[n]\setminus S}=0$ and  $F^\top \bar x=0$, then for any $(x,z,t)\in X$ we find that
$$\lim_{\lambda\to 0^+}(1-\lambda)(x,z,t)+\lambda ((1/\lambda)\bar x,\hat\ones_S,0)=(x+\bar x, z,t)\in \clconv(X).$$
In particular, the point $(x+\bar x,z,t)$, which may not satisfy the complementarity constraints, cannot be separated from $\clconv(X)$, or any closed relaxation. 
On the other hand, if matrix $Q$ is full-rank, then $F^\top \bar x=0\implies \bar x=0$ (as shown in the proof of Proposition~\ref{prop:valid}); therefore, the complementarity constraints are enforced by the conic constraint.  

Recall that $\pi_S : \R^{n} \rightarrow \R^{S}$ is the projection onto the subspace indexed by $S$. Now we consider the natural convex relaxation of \eqref{eq:miqoSDPSep2} \revised{by dropping constraint \eqref{eq:miqoSDPSep2_feasible}, and show that it is ideal under certain technical conditions over $F$ and the set $Z$, as stated in Theorem ~\ref{thm:separable} below.}

\revised{
\begin{theorem}\label{thm:separable}
Let $Q = F F^{\top}$, where $F \in \R^{n \times k}$ is a full-column rank matrix satisfying $\mathrm{col}(F) = \bigcap_{\hat \ones_{S} \in Z} \pi_S^{-1}(\mathrm{col}(F_S))$. Then
\begin{equation*}
\clconv(X) \; = \; \{(z, x, t) \in [0,1]^{n} \times \R^{n+1} \; | \; \exists W \in \R^{k \times k} \; \text{s.t.} \ \eqref{eq:miqoSDPSep2_psd}, \eqref{eq:miqoSDPSep2_polytope}\}.
\end{equation*}
\end{theorem}}

\begin{proof}
\DIFdelbegin 
\DIFdelend \DIFaddbegin \revised{
	Clearly, constraints \eqref{eq:miqoSDPSep2_psd},\eqref{eq:miqoSDPSep2_polytope} define a closed convex set.}
\DIFaddend 
Consider the two optimization problems:
\begin{subequations}\label{opt:X}
\begin{align}
\min \quad & a^{\top} x + b^{\top} z + \tfrac{1}{2}t \\
\text{s.t.} \; \; & (x,z,t) \in \clconv(X),  
\end{align}
\end{subequations}
and
\begin{subequations}\label{opt:Y}
\begin{align}
\min \quad & a^{\top} x + b^{\top} z + \tfrac{1}{2} t \\
\text{s.t.} \; \; & \begin{pmatrix}W& F^\top x \\
x^\top F  &t\end{pmatrix} \succcurlyeq \mathbf 0,\\
& (z,W)\in P_F,\; x\in \R^n, t\in \R. 
\end{align}	

\end{subequations}

It suffices to show that problem \eqref{opt:X} and \eqref{opt:Y} always attain the same optimal value. Consider the following two cases:

\vskip 1mm
$\bullet$ $F F^{\dagger} a \neq a$:
In other words, $a$ is not in the column space of $F$, i.e., $a \notin \mathrm{col}(F)$. In this case, by the condition $\mathrm{col}(F) = \bigcap_{\hat \ones_{S} \in Z} \pi_{S}^{-1}(\mathrm{col}(F_S))$, there exists one $\hat \ones_{S} \in Z$ such that $a_{S} \notin \mathrm{col}(F_S)$. Then, let $z$ be such that $z_i = 1, \; \forall i \in S$. \revised{Since $a_{S} \notin \mathrm{col}(F_{S})$, there exists $x$ such that $x_i = 0$ for all $i \in [n]\backslash S$, $x_{S}$ is in the orthogonal complement of $F_S$ and  $a_S^\top x_S < 0$.}
Clearly, $z$ and $x$ satisfy the constraint $x_i (1 - z_i) = 0$ for all $i = 1, \dots,n$. Complementarity holds for $\lambda x$ for $\lambda > 0$ as well. Since,
by construction, $x^{\top} F F^{\top} x = 0$, the objective  $b^\top z + \lambda \langle a, x \rangle + \lambda^2 (x^{\top} F F^{\top} x)$  tends to $-\infty$ for $(\lambda x, z)$ as $\lambda \to \infty$. Thus problem \eqref{opt:X} is unbounded and since problem \eqref{opt:Y} is a convex relaxation of \eqref{opt:X}, problem \eqref{opt:Y} is unbounded as well.

\vskip 1mm
$\bullet$ $F F^{\dagger} a = a$:
For problem (\ref{opt:Y}), we can project out $t$ using the relation
$$
\begin{pmatrix}
W & F^\top x   \\
x^{\top} F & t
\end{pmatrix} \succcurlyeq 0 \quad \text{iff} \quad W W^{\dagger} F^{\top} x = F^{\top} x \; \; \text{and} \; \; 
t \geq x^{\top} F W^{\dagger} F^{\top} x.
$$
Therefore, problem (\ref{opt:Y}) is equivalent to
\begin{subequations}\label{opt:Y_1}
\begin{align}
\min \quad & a^{\top} x + b^{\top} z + \tfrac{1}{2} x^{\top} F W^{\dagger} F^{\top} x \\
\text{s.t.} \; \; 
& W W^{\dagger} F^{\top} x = F^{\top} x  \\
& (z,W)\in P_F,\; x\in \R^n. 
\end{align}	
\end{subequations}

Since $F F^{\dagger} a = a$, we can write $a^{\top} x = (F^{\dagger} a)^{\top} F^{\top} x$. Define $\tilde a = F^{\dagger} a$, then $a^{\top} x = \tilde a^{\top} F^{\top} x$. Substituting $F^{\top} x$ with a new variable $u \in \R^{k}$ and since $F$ has full column rank, problem \eqref{opt:Y_1} is equivalent to 
\begin{subequations}\label{opt:Y_2}
\begin{align}
\min \quad & b^{\top} z + \tilde a^{\top} u + \tfrac{1}{2} u^{\top} W^{\dagger} u \\
\text{s.t.} \; \; 
& W W^{\dagger} u = u  \\
& (z,W)\in P_F, u\in \R^k. 
\end{align}	
\end{subequations}

Using identical arguments as in the proof of Theorem~\ref{thm:canonical}, we find that there exists $\hat \ones_{S}\in Z$ such that $(u^*,z^*,W^*)=(-\hat F_S^\dagger \hat F_S \tilde a, \hat \ones_S, \hat F_S^\dagger \hat F_S)$ is optimal for \eqref{opt:Y_2}. We now construct an optimal solution for \eqref{opt:Y_1}. Let $x^*$ be defined as $x_S^*=-(F_{S}^{\dagger})^{\top} F_{S}^{\dagger} a_{S}$ and $x_{[n]\setminus S}^*=0$, and observe that $(x^*,z^*)$ is feasible for \eqref{opt:X}, with objective $\sum_{i \in S} b_i - \frac{1}{2} \| F_{S}^{\dagger} a_{S} \|_2^2$. 
Substituting $W^{*} = \hat F_{S}^{\dagger} \hat F_{S}$, the optimal value of problem \eqref{opt:Y} equals 
$\sum_{i \in S} b_i - \frac{1}{2}  \|F_{S}^{\dagger} F_{S} F^{\dagger} a \|_2^2$.
Note that both $\alpha_1 = F^{\dagger} a$ and $\alpha_2 = F_{S}^{\dagger} a_{S}$ satisfy the equation $F_{S} \alpha = a_{S}$ and thus $\alpha_1 - \alpha_2$ is orthogonal to the row space of $F_{S}$ which means $F_{S}^{\dagger} F_{S} \alpha_1 = F_{S}^{\dagger} F_{S} \alpha_2 = \alpha_2$. Hence, we conclude that the optimal values of problem \eqref{opt:X} and problem \eqref{opt:Y} coincide.
\end{proof}

\begin{remark}
	\revised{From the first case analysis of the proof of Theorem~\ref{thm:separable}, one sees that the technical condition $\mathrm{col}(F) = \bigcap_{\hat \ones_{S} \in Z} \pi_S^{-1}(\mathrm{col}(F_S))$ is equivalent to stating that the mixed-integer optimization problem and the proposed convex relaxation are unbounded at the same time.} The condition is automatically satisfied if $\ones \in Z$.
	 Moreover, if matrix $Q$ is rank-one, then this condition is equivalent to the nondecomposability condition on $Z$ given in \cite{wei2021ideal}. 
	If it fails to hold, the convexification presented is still valid but may be weak: the convex relaxation may be unbounded even if the mixed-integer optimization problem is bounded. We provide an example illustrating this phenomenon in \S\ref{sec:technicalCondition}.  \qed
\end{remark}

\begin{remark}
	An immediate consequence of Theorem~\ref{thm:separable} is that if matrix $Q$ is rank-deficient, i.e., $k<n$, 	then the extended formulation describing $\clconv(X)$ is simpler than the full rank case, i.e., it has fewer additional variables and lower-dimensional conic constraints.\qed
\end{remark}

We now illustrate Theorem~\ref{thm:separable} by providing an alternative proof of the main result of \cite{atamturk2019rank} using our unifying framework. 
\subsubsection{Rank-one quadratic functions} \label{sec:rank-one}

Consider the rank-one set $$
X_{R1}=\left\{(x,z,t)\in \R^n\times \{0,1\}^n\times \R: t\geq \left(h^\top x\right)^2,\; x\circ (\ones-z)=0\right\},$$
where we assume $h_i\neq 0$ for all $i\in [n]$.

\begin{proposition}[\cite{atamturk2019rank}]\label{prop:rankOne}
The closure of the convex hull of $X_{R1}$ is
	\begin{align*}
	\clconv(X_{R1})=&\Bigg\{(x,z,t)\in \R^{2n+1}:
	\begin{pmatrix}\min\{1,\ones^\top z\}&h^\top x\\
	h^\top x&t\end{pmatrix}\succeq 0,\;0\leq z\leq \ones\Bigg\}.
	\end{align*}
\end{proposition}
\begin{proof}
	In the case of a rank-one function, we have $F=h$ and $W \in \R^1$. Note that the pseudoinverse of vector $\hat h_S$ is given by $$\hat h_S^\dagger=\begin{cases}0&\text{if }\hat h_S=0\\
	\hat h_S^\top/(\hat h_S^\top \hat h_S)&\text{otherwise,}\end{cases}$$
	and, in particular, we find that $\hat h_S^\dagger \hat h_S=1$ if $S\neq \emptyset$, and $\hat h_S^\dagger \hat h_S=0$ otherwise. Thus, $\hat h_S^\dagger \hat h_S=\max\{z_1,\dots,z_n\}$, and  $P_F$ is described by the linearization $0\leq W\leq \min\{1,\ones^\top z\}$. Projecting out variable $W$, we arrive at the result. 
\end{proof}

We discuss generalizations of $X_{R1}$ with arbitrary constraints on the indicator variables in Section~\ref{sec:original}.

\subsubsection{An example with a rank-two quadratic function}

In order to illustrate how convexification methods for polyhedra can be directly utilized to convexify the mixed-integer nonlinear set $X$, we consider 
a special rank-two quadratic function with three variables and the associated set
$$
X_{3}=\left\{(x,z,t)\in \R^3\times \{0,1\}^3\times \R: t\geq (x_1+x_2+x_3)^2+x_3^2,\; x\circ (\ones-z)=0\right\}.
$$
In this case, $Q = F F^\top$ with  $F^\top= \small \begin{pmatrix}1 &  1 &1 \\
0&0 & 1 \end{pmatrix}$. 
The extreme points of $P_F$ are given in Table~\ref{tab:extrX3}.
 Using PORTA \cite{PORTA} to switch from the extreme point representation  of $P_F$ to its 
 facial description, we obtain the closure of the convex hull of $X_{3}$:
	\begin{align*}
	\clconv(X_{3})=&\Bigg\{(x,z,t)\in \R^7:\exists W\in \R^{2\times 2}\text{ such that }\\
	&\small \begin{pmatrix}W_{11}&W_{12}&x_1+x_2+x_3\\
	W_{12}&W_{22}&x_3\\
	x_1+x_2+x_3&x_3&t\end{pmatrix}\succeq 0,\\
	&z_3=W_{12}+W_{22},\;0\leq W_{12}\leq W_{22}\leq W_{11},\\
	&z_3+\max\{z_1,z_2\}\leq W_{11}+W_{22}\leq z_1+z_2+z_3,\\
	&W_{11}+2W_{12}+W_{22}\leq 1+z_3\Bigg\}.
	\end{align*}

\begin{table}[!h]
	\caption{Extreme points of $P_F$ corresponding to set $X_3$.} 
	\label{tab:extrX3}
	\vskip 1mm
	\centering \small
	\begin{tabular}{c | c | c| c} 
		\hline \hline
		$z$ & $\hat F_S^\top$ & $\hat F_S^\dagger$ & $\hat F_S^\dagger \hat F_S$ \\
		\hline
		$(0,0,0)$ &\small $\begin{pmatrix}0 & 0 & 0\\
			0&0 & 0\end{pmatrix}$&$\begin{pmatrix}0 & 0&0\\0&
			0&0\end{pmatrix}$&$\begin{pmatrix}0 & 0\\0&
			0\end{pmatrix}$\\
		$(0,0,1)$ &\small $\begin{pmatrix}0 & 0&1\\
			0 & 0 &1\end{pmatrix}$&$\begin{pmatrix}0 & 0&1/2\\0&
			0&1/2\end{pmatrix}$&$\begin{pmatrix}1/2 & 1/2\\1/2&
			1/2\end{pmatrix}$\\
		$(0,1,0)$ &\small $\begin{pmatrix}0 & 1&0\\
			0&0 & 0\end{pmatrix}$&$\begin{pmatrix}0 & 1&0\\0&
			0&0\end{pmatrix}$&$\begin{pmatrix}1 & 0\\0&
			0\end{pmatrix}$\\
		$(0,1,1)$ &\small $\begin{pmatrix}0 &1&1\\
			0 & 0 &1\end{pmatrix}$&$\begin{pmatrix}0 & 1&0\\0&
			-1&1\end{pmatrix}$&$\begin{pmatrix}1 & 0\\0&
			1\end{pmatrix}$\\
		$(1,0,0)$ &\small $\begin{pmatrix}1 & 0&0\\
			0 & 0&0\end{pmatrix}$&$\begin{pmatrix}1 & 0&0\\0&
			0&0\end{pmatrix}$&$\begin{pmatrix}1 & 0\\0&
			0\end{pmatrix}$\\
		$(1,0,1)$ &\small $\begin{pmatrix}1 & 0&1\\
			0& 0&1\end{pmatrix}$&$\begin{pmatrix}1 & 0&0\\-1&
			0&1\end{pmatrix}$&$\begin{pmatrix}1 & 0\\0&
			1\end{pmatrix}$\\
		$(1,1,0)$ &\small $\begin{pmatrix}1 & 1&0\\
			0&0&0\end{pmatrix}$&$\begin{pmatrix}1/2 & 1/2&0\\0&
			0&0\end{pmatrix}$&$\begin{pmatrix}1 & 0\\0&
			0\end{pmatrix}$\\
		$(1,1,1)$ &\small $\begin{pmatrix}1 & 1&1\\
			0 & 0 &1\end{pmatrix}$&$\begin{pmatrix}1/2 & 1/2&0\\-1/2&
			-1/2 &1\end{pmatrix}$&$\begin{pmatrix}1 & 0\\0&
			1\end{pmatrix}$\\ \hline \hline
	\end{tabular}
\end{table}

\subsubsection{An example where the technical condition fails}
\label{sec:technicalCondition}
	Consider the set $$
	X_{R1}^{C1} \!=\!\left\{\!(x,z,t)\in \R^n \! \times \! \{0,1\}^n \! \times \! \R: t\geq \left(h^\top x\right)^2, x\circ (\ones-z)=0, \sum_{i=1}^n z_i\leq 1\right\}$$
	with $h_i\neq 0$ for $i\in [n]$. 
	In this case, $F=h$ and $\mathrm{col}(F_{\{i\}})=\R$ and $\pi_S^{-1}(\mathrm{col}(F_{\{i\}}))=\R^n$. Thus, $  \bigcap_{\hat \ones_{S}\in Z} \pi_S^{-1}(\mathrm{col}(F_S))=\R^n$, while $\mathrm{col}(F) = \{x\in \R^n: x=\lambda h \text{ for some }\lambda \in \R \}$, and the technical assumption is not satisfied. 
	
	The relaxation induced by \eqref{eq:miqoSDPSep2_psd}, \eqref{eq:miqoSDPSep2_polytope}, \eqref{eq:miqoSDPSep2_x}, which is constructed as outlined in Proposition~\ref{prop:rankOne}, results in the set induced by bound constraints $0\leq z\leq 1$, $\ones^\top z\leq 1$ and
	$t\geq (h^\top x)^2/(\ones^\top z)$. Moreover, the corresponding optimization problem
	$$\min_{x,z}\;a^\top x+ b^\top z+(h^\top x)^2/(\ones^\top z) \text{ s.t. }\ones^\top z\leq 1,\; x\in \R^n,\; z\in [0,1]^n$$
	is unbounded unless $a\in \mathrm{col}(F)$. 
	
	In contrast, $\clconv(X_{R1}^{C1})$ is described via constraint $t\geq \sum_{i=1}^n h_i^2x_i^2/z_i$ \cite{wei2020convexification,wei2021ideal} (similar to the result described in \S\ref{sec:choose1}), and the corresponding optimization problem is always bounded. 

\section{Convexification in the original space}\label{sec:original}

We now turn our attention to describing $\clconv(X)$ in the original space of variables. The discussion of this section is based on projecting out the matrix variable $W$ in the canonical description of $\clconv(X)$ given in Theorem~\ref{thm:canonical} for $Q\succ 0$. Identical arguments hold for the representation in Theorem~\ref{thm:separable} for low-rank matrices.

Suppose that a minimal description of polyhedron $P$ 
is given by the facet-defining inequalities 
\begin{equation}\langle \Gamma_i,W \rangle-\gamma_i^\top z\leq \beta_i, \quad i=1,\dots,m_1,\label{eq:facets}\end{equation}
and equalities
\begin{equation*}\langle \Gamma_i,W \rangle-\gamma_i^\top z= \beta_i, \quad i=m_1+1,\dots,m,\label{eq:faces}\end{equation*}
where $\Gamma_i\in \R^{n \times n},$ $\beta_i\in \R$ and $\gamma_i\in \R^n$. Theorem~\ref{thm:projection} describes $\clconv(X)$ in the original space of variables. \revised{Note that, in practice, a complete description may not be explicitly available, in which case one can use a partial description to derive valid inequalities.}

\revised{
Before we give the description in the original space, we define a set of feasible coefficients used to derive the inequalities. Let
$$\mathcal Y \defeq \bigg \{y\in \R_+^{m_1}\times \R^{m-m_1}: \sum_{i=1}^m \Gamma_iy_i\succeq 0,\;\sum_{i=1}^m \text{Tr}(\Gamma_i)y_i\leq 1 \bigg \}.$$

\begin{theorem}\label{thm:projection}
	If $Q\succ 0$, point $(x,z,t)\in \clconv(X)$ if and only if $z\in \conv(Z)$, $t\geq 0$ and
	\begin{align}
	 t\geq \;&\frac{x^\top \left(\sum_{i=1}^m \Gamma_iy_i\right)x}{y^\top \beta+\left(\sum_{i=1}^my_i\gamma_i\right)^\top z}, & \forall y\in \mathcal Y, \label{eq:semiinfSocp}
	\end{align}
 or equivalently,
	\begin{align}
	 t\geq \max_{y\in \mathcal Y}\;&\frac{x^\top \left(\sum_{i=1}^m \Gamma_iy_i\right)x}{y^\top \beta+\left(\sum_{i=1}^my_i\gamma_i\right)^\top z} \cdot \label{eq:infSocp-c}
	\end{align}
\end{theorem}
 }
 
\begin{proof}
	A point $(x,z,t)\in\clconv(X)$ if and only if
	\begin{align*}
	0\geq \min_{W,\lambda} \;&\lambda\\
	\text{s.t.}\;&\langle \Gamma_i,W \rangle\leq \beta_i+\gamma_i^\top z, \ \  i=1,\dots,m_1\\
	&\langle \Gamma_i,W \rangle = \beta_i+\gamma_i^\top z, \ \  i=m_1+1,\dots,m\\
	&W-xx^\top/t+\lambda I\succeq 0, \ \lambda\geq 0.  
	\end{align*}
	Strong duality holds since there exists $(z,W)\in P$ that satisfies the facet-defining inequalities strictly, and we can always increase $\lambda$ to find a strictly feasible solution to the above minimization problem.
	Substituting $V=W-xx^\top/t+\lambda I$, the optimization problem simplifies to 
	\begin{subequations}
	\begin{align*}
	0\geq \min_{V,\lambda} \;&\lambda\\
	\text{s.t.}\;&-\langle \Gamma_i,V \rangle+\lambda\text{Tr}(\Gamma_i)\geq -\beta_i-\gamma_i^\top z+\langle \Gamma_i,xx^\top /t\rangle, \ i=1,\dots,m_1\tag{$y_i$}\\
	&-\langle \Gamma_i,V \rangle+\lambda\text{Tr}(\Gamma_i)= -\beta_i-\gamma_i^\top z+\langle \Gamma_i,xx^\top /t\rangle, \ i=m_1+1,\dots,m\tag{$y_i$}\\	
	&V\succeq 0, \   \lambda\geq 0.
	\end{align*}
	\end{subequations}
		Letting $y\in  \R_+^{m_1}\times \R^{m-m_1}$ denote the dual variables, we find the equivalent representation
	\begin{subequations}
		\begin{align}
		0\geq \max_{y\in  \R_+^{m_1}\times \R^{m-m_1}}\;&\sum_{i=1}^m y_i\left(-\beta_i-\gamma_i^\top z+\langle \Gamma_i, xx^\top/t\rangle\right)\label{eq:dual_obj}\\
		\text{s.t.}\;&-\sum_{i=1}^m y_i\Gamma_i\preceq 0,\; \sum_{i=1}^m\text{Tr}(\Gamma_i)y_i\leq 1.   \label{eq:dual_constr}
		\end{align}
	\end{subequations}
	In particular, inequality \eqref{eq:dual_obj} is valid for any fixed feasible $y$. Multiplying both sides of the inequality by $t$, we find the equivalent conic quadratic representation
	\begin{equation}\label{eq:socp}t\left(y^\top \beta+\left(\sum_{i=1}^my_i\gamma_i\right)^\top z\right)\geq \langle \sum_{i=1}^m y_i\Gamma_i,xx^\top\rangle.\end{equation}
	Note that validity of inequalities \eqref{eq:socp} implies that $y^\top \beta+\left(\sum_{i=1}^my_i\gamma_i\right)^\top z\geq 0$ for any primal feasible $z$ and dual feasible $y$;
	dividing  both sides of the inequality by $y^\top \beta+\left(\sum_{i=1}^my_i\gamma_i\right)^\top z$, the theorem is proven. 
\end{proof}

Note that even if inequalities \eqref{eq:facets} are not facet-defining or are insufficient to describe $P$, the corresponding inequalities \eqref{eq:infSocp-c} are still  valid  for $\clconv(X)$.

\revised{
We also state the analogous result for low-rank matrices, without proof, where $(\Gamma_i,\gamma_i,\beta_i), i\in [m]$ defines $P_F$.
\begin{theorem}\label{thm:separable-orig}
Let $Q = F F^{\top}$, where $F \in \R^{n \times k}$ is a full-column rank matrix satisfying $\mathrm{col}(F) = \bigcap_{\hat \ones_{S} \in Z} \pi_S^{-1}(\mathrm{col}(F_S))$. 
	Then point $(x,z,t)\in \clconv(X)$ if and only if $z\in \conv(Z)$, $t\geq 0$ and
	\begin{align}
	 t\geq \;&\frac{x^\top F \left(\sum_{i=1}^m \Gamma_iy_i\right)  F^\top x}{y^\top \beta+\left(\sum_{i=1}^my_i\gamma_i\right)^\top z}, & \forall y\in \mathcal Y, \label{eq:semiinfSocp}
	\end{align}
 or equivalently,
	\begin{align}
	 t\geq \max_{y\in \mathcal Y}\;&\frac{x^\top F \left(\sum_{i=1}^m \Gamma_iy_i\right)F^{\top}x}{y^\top \beta+\left(\sum_{i=1}^my_i\gamma_i\right)^\top z} \cdot \label{eq:infSocp-c}
	\end{align}
\end{theorem}}

We now  illustrate Theorem~\ref{thm:projection} for the set $X_{2\times 2}$ discussed in \S\ref{sec:2x2}. 

\begin{example}[Description of $\clconv(X_{2\times 2})$ in the original space]\label{ex:2x2}
	From Proposition~\ref{prop:2x2extended}, we find that for $X_{2\times 2}$, a minimal description of  polyhedron $P$ is given  by the 
	bound constraints $0\leq z\leq 1$ and 
	\begin{align*} \small
	\left\langle\begin{pmatrix}1 & -1/(2d_1)\\-1/(2d_1)&0\end{pmatrix},W\right\rangle -(1/d_1)z_1&=0\tag{$y_1$}\\ \small
	\left\langle\begin{pmatrix}0 & -1/(2d_2)\\-1/(2d_2)&1\end{pmatrix},W\right\rangle -(1/d_2)z_2&=0\tag{$y_2$}\\ \small
	\left\langle\begin{pmatrix}0 & -1/2\\-1/2&0\end{pmatrix},W\right\rangle&\leq 0\tag{$y_3$}\\ \small
	\left\langle\begin{pmatrix}0 & -1/2\\-1/2&0\end{pmatrix},W\right\rangle+(1/\Delta)z_1+(1/\Delta)z_2&\leq 1/\Delta\tag{$y_4$}\\ \small
	\left\langle\begin{pmatrix}0 & 1/2\\1/2&0\end{pmatrix},W\right\rangle-(1/\Delta)z_1&\leq 0\tag{$y_5$}\\
	\small
	\left\langle\begin{pmatrix}0 & 1/2\\1/2&0\end{pmatrix},W\right\rangle-(1/\Delta)z_2&\leq 0.\tag{$y_6$}\\
	\end{align*}
	\end{example}
	
	Then, an application of Theorem~\ref{thm:projection} yields the inequality
	\begin{subequations}
	\label{eq:original2x2}
	\begin{align}
	t\geq\max_{y\in \R_+^6}\;&\frac{y_1x_1^2+y_2x_2^2+(-y_1/d_1-y_2/d_2-y_3-y_4+y_5+y_6)x_1x_2}{(1/\Delta)y_4+(y_1/d_1-y_4/\Delta+y_5/\Delta)z_1+(y_2/d_2-y_4/\Delta+y_6/\Delta)z_2}\\
	\text{s.t.}\;&4y_1y_2\geq (-y_1/d_1-y_2/d_2-y_3-y_4+y_5+y_6)^2,\; y_1+y_2\leq 1.
	\end{align}
	\end{subequations}
Note that variables $y_1,y_2$ are originally free as dual variables for equality constraints, however, the nonnegativity constraints are imposed due to the positive definiteness constraint in $\mathcal Y$. 
In Appendix~\ref{app:validity} we provide an independent verification that inequality \eqref{eq:original2x2} is indeed valid, and reduces to the quadratic inequality $t\geq d_1x_1^2+d_2x_2^2-2x_1x_2$ at integral $z$.\qed

From Theorem~\ref{thm:projection}, we see that $\clconv(X)$ can be described by an infinite number of fractional quadratic/affine inequalities \eqref{eq:infSocp-c}.  More importantly, the convex hull is finitely generated: the infinite number of quadratic and affine functions are obtained from conic combinations of a \emph{finite} number of base matrices $\Gamma_i$ and vectors $(\gamma_i,\beta_i)$, which correspond precisely to the minimal description of $P$. \revised{To solve the resulting semi-infinite problem in practice, one can employ a delayed cut generation scheme, where at each iteration, the problem with a subset of  inequalities  \eqref{eq:semiinfSocp} is solved to obtain $(\bar x,\bar z)$. Then, the separation problem  to find a maximum violated inequality (i.e., $y$) at $(\bar t, \bar x,\bar z)$, if it exists,  is a convex optimization problem given by the inner maximization problem in \eqref{eq:infSocp-c}.}

\begin{example}[Rank-one function with constraints] Given $Z\subseteq \revised{\{0,1\}^{n}}$, consider the set 
	$$X_{R1}^Z=\left\{(x,z,t)\in \R^n\times Z\times \R: t\geq \left(h^\top x\right)^2,\; x\circ (\ones-z)=0\right\},$$ that is, a rank-one function with arbitrary constraints on the indicator variables $z$ defined by $Z$. As discussed in the proof of Proposition~\ref{prop:rankOne}, $P_F\subseteq \R^{n+1}$ with one additional variable $W \in \R^1$ which, at integer points, is given by $W=\max\{z_1,\dots,z_n\}$.  For simplicity, assume that $0\in Z$, and that both $\conv(Z)$ and $\conv(Z\setminus \{0\})$ are full-dimensional. Finally, consider all facet-defining inequalities of $\conv(Z\setminus \{0\})$ of the form $\gamma_i^\top z\geq 1$ (that is, inequalities that cut off point $0$), for $i=1,\dots,m$. Now consider inequalities
	\begin{equation}
	\label{eq:facetDefiningZ}
	W\leq \gamma_i^\top z,\qquad \forall i\in [m].
	\end{equation}
First, observe that inequalities \eqref{eq:facetDefiningZ} are valid for $P_F$: given $z\in Z$, if $z=0$, then $W=0$; otherwise, $z\in Z\setminus\{0\}\implies \gamma_i^\top z\geq 1=W$. Second, note that inequalities \eqref{eq:facetDefiningZ} are facet-defining for $P_F$. Indeed, given $i\in [m]$, consider the face $Z_i=\{z\in \conv(Z\setminus \{0\}): \gamma_i^\top z=1\}$ of $\conv(Z\setminus\{0\})$: since $\conv(Z\setminus\{0\})$ is full-dimensional and $\gamma_i^\top z\geq 1$ is facet-defining, there are $n$ affinely independent points $\{z^j\}_{j=1}^n$  such that $z^j\in Z_i$. Thus, we find that points $(z^j,1)_{j=1}^n$ and $(0,0)$ are ($n+1$)-affinely independent points satisfying \eqref{eq:facetDefiningZ} at equality. Moreover, one can easily verify that inequality $W\leq 1$ is facet-defining as well. Thus, from \eqref{eq:infSocp-c} (adapted to the factorable representation discussed in \S\ref{sec:separable}), we conclude that the inequality
\begin{align}\label{eq:r1solving}
\displaystyle t\geq \max_{y\in \R_+^{m+1}}\;& \left \{ \frac{\left(\sum_{i=0}^m y_i\right)(h^\top x)^2}{y_0+\sum_{i=1}^my_i(\gamma_i^\top z)} \ \
\text{s.t. }\;\sum_{i=0}^m y_i\leq 1 \right \}
\end{align}
is valid for $\clconv(X_{R1}^Z)$. Moreover, an optimal solution to optimization problem \eqref{eq:r1solving} corresponds to setting $y_i=1$ for  $i\in\arg\min_{i\in [m]}\{\gamma_i^\top z\}$, and we conclude that inequalities $t\geq (h^\top x)^2$ and $t\geq (h^\top x)^2/(\gamma_i^\top z), i\in [m]$ are valid for $\clconv(X_{R1}^Z)$. Indeed, as shown in \cite{wei2021ideal}, these inequalities along with $z\in\conv(Z)$ fully describe $\clconv(X_{R1}^Z)$ (when a nondecomposability condition holds).  \qed
\end{example}

\subsection*{Connection with decomposition methods}

From  \revised{Theorem ~\ref{thm:projection}}, we see that the convex hull, \revised{$X$,} is obtained by  \revised{adding conic quadratic inequalities  $ t \ge \frac{x^\top \left(\sum_{i=1}^m \Gamma_iy_i\right)x}{y^\top \beta+\left(\sum_{i=1}^my_i\gamma_i\right)^\top z}$ with simpler quadratic structure $x^{\top} \Gamma_i x$} (corresponding to inequalities describing $P$). In particular, the intuition is similar to convexifications obtained from decompositions \eqref{eq:decomp}. We now show how the theory presented in this paper sheds light on the strength of the aforementioned decompositions. 

Suppose inequalities \eqref{eq:facets}, which we repeat for convenience: 
\begin{equation}\langle \Gamma_i,W \rangle-\gamma_i^\top z\leq \beta_i, \quad i=1,\dots,m,\label{eq:facetsRepeated}\end{equation} are valid for $P$ and, additionally, $\Gamma_i\succeq 0$ for all $i\in [m]$. Since $P$ is not full-dimensional in general, positive semidefiniteness conditions may not be as restrictive as they initially seem.

\begin{example}[Description of $\clconv(X_{2\times 2})$, continued]\label{ex:psd} None of the matrices in the facets of $P$ for $\clconv(X_{2\times 2})$ given in Example~\ref{ex:2x2} are positive semidefinite. Nonetheless, the inequalities below also describe $P$ (we abuse notation and encode using variables $y$ how each inequality is obtained):
	\begin{align*} \small
	\left\langle\begin{pmatrix}1 & -1/d_1\\-1/d_1&d_2/d_1\end{pmatrix},W\right\rangle -(1/d_1)(z_1+z_2)&=0\tag{$y_1+(d_2/d_1)y_2$}\\ \small
	\left\langle\begin{pmatrix}d_1/d_2 & -1/d_2\\-1/d_2&1\end{pmatrix},W\right\rangle -(1/d_2)(z_1+z_2)&=0\tag{$y_2+(d_1/d_2)y_1$}\\ \small
	\left\langle\begin{pmatrix}d_1/2 & -1\\-1&d_2/2\end{pmatrix},W\right\rangle-(1/2)(z_1+z_2)&\leq 0\tag{$y_3+(d_1/2)y_1+(d_2/2)y_2$}\\ \small
	\left\langle\begin{pmatrix}d_1/2 & -1\\-1&d_2/2\end{pmatrix},W\right\rangle+(1/\Delta-1/2)z_1+(1/\Delta-1/2)z_2&\leq 1/\Delta\tag{$y_4+(d_1/2)y_1+(d_2/2)y_2$}\\ \small
	\left\langle\begin{pmatrix}d_1 & 0\\0&0\end{pmatrix},W\right\rangle-(d_1d_2/\Delta)z_1&\leq 0\tag{$y_5+d_1y_1$}\\ \small
	\left\langle\begin{pmatrix}0 & 0\\0&d_2\end{pmatrix},W\right\rangle-(d_1d_2/\Delta)z_2&\leq 0\tag{$y_6+d_2y_2$}\\
	\end{align*}
In particular, the last two inequalities satisfy positive semidefiniteness. Moreover, the relaxation of the first two equalities obtained by replacing them with  inequalities also satisfies positive semidefiniteness. Finally, if $Q$ is sufficiently diagonally dominant and $d_1d_2\geq 4$, then the third and fourth inequalities satisfy positive semidefiniteness as well.
\qed
\end{example}

Now suppose that in \eqref{eq:infSocp-c}, we fix $y_i=\lambda/(\beta_i+\gamma_i^\top z)$, where $\lambda$ is small enough to ensure that constraint $\sum_{i=1}^m\tr(\Gamma_i)y_i\leq 1$ is satisfied. Then inequality \eqref{eq:infSocp-c} reduces to
$$mt\geq \sum_{i=1}^m \frac{x^\top \Gamma_i x}{\beta_i+\gamma_i^\top z},$$
which is precisely the relaxations obtained from \eqref{eq:decomp}. We make the following two important observations.

\vskip 1mm
\paragraph{\textit{Observation 1}} Relaxations obtained by fixing a given decomposition \eqref{eq:decomp} \cite{Frangioni2007,frangioni2018decompositions} are, in general, \emph{insufficient} to describe $\clconv(X)$. Indeed, from Theorem~\ref{thm:projection}, describing $\clconv(X)$ requires one inequality per extreme point of the region $\mathcal Y$,
whereas a given decomposition corresponds to a single point in this region. 

\vskip 1mm
\paragraph{\textit{Observation 2}} On the other hand, the strong ``optimal" or ``dynamic" relaxations \cite{zheng2014improving,dong2015regularization,atamturk2019rank}, where the decomposition is not fixed but instead is chosen dynamically, are \emph{excessive} to describe $\clconv(X)$. Indeed, they are of the form \eqref{eq:infSocp-c} for every possible (rank-one, $2\times2$, remainder) matrix, and are not finitely generated\revised{; whereas, our results imply that the necessary inequalities are finitely generated.} 

We conclude this section with an analysis of rank-one decompositions, where we assume for simplicity that $Q\succ 0$: given a subset $\mathcal{T}\subseteq 2^{[n]}$, rank-one relaxations are given by
\begin{equation}t\geq\sum_{T\in \mathcal{T}}\frac{(\hat h_T^\top x)^2}{\hat\ones_T^\top z }+x^\top Rx,\label{eq:rankOneDecomp}\end{equation}
where $R=Q-\sum_{T\in \mathcal{T}}\hat h_T\hat h_T^\top\succeq 0$, and $\hat h_T\in \R^n$ are given vectors that are zero in entries not indexed by $T$. Relaxation \eqref{eq:rankOneDecomp} can \revised{be} interpreted as a decomposition obtained from valid inequalities for $P$ of the form
\begin{equation}\label{eq:validRankone}\langle\hat h_T\hat h_T^\top,W \rangle\leq \gamma\hat\ones_T^\top z,\end{equation}
where $\gamma\geq 0$.
Note that inequality \eqref{eq:validRankone} is valid for $P$ if \begin{equation}\label{eq:ub}\gamma\geq \max_{\hat \ones_{S}\in Z} \frac{1}{|S \bigcap T|}\langle \hat h_T\hat h_T^\top, \hat Q_S^{-1} \rangle.\end{equation} 
\begin{proposition}
	 If $\gamma= \max_{\hat \ones_{S}\in Z} \frac{1}{|S \bigcap T|} \langle \hat h_T\hat h_T^\top, \hat Q_S^{-1} \rangle$, then inequality \eqref{eq:validRankone} defines a face of $P$ of dimension at least $\mathrm{dim}(P_0)+1$, where $$P_0=\left\{(z,W)\in P: z_T=0\text{ and } W_T=0\right\}.$$
\end{proposition}
\begin{proof}
There are $\text{dim}(P_0)+1$ affinely independent points in $P_0$, and all satisfy \eqref{eq:validRankone} at equality. Letting $S^*\in \argmax_{\hat \ones_{S}\in Z} \frac{1}{|S \bigcap T|} \langle \hat h_T\hat h_T^\top, \hat Q_S^{-1} \rangle$, we find that $(\hat \ones_{S^*}, \hat Q_{S^*}^{-1})$ is an additional affinely independent point satisfying \eqref{eq:validRankone} at equality.  
\end{proof}
Note that if optimization problem \eqref{eq:ub} has multiple optimal solutions, then one can find additional affinely independent points.
In particular, \eqref{eq:validRankone} is guaranteed to define a high dimensional face of $P$ if $|T|$ is small. Indeed, inequalities \eqref{eq:validRankone} were found to be particularly effective computationally if $\mathcal{T}=\left\{T\subseteq [n]: |T|\leq \kappa\right\}$ for some small $\kappa$ \cite{atamturk2019rank}, although a theoretical justification of this observation has been missing until now.

	\begin{remark}\revised{[Description of $\clconv(X_{2\times 2})$, continued]} Consider again the facet-defining inequalities given in Example~\ref{ex:psd}. The last two inequalities correspond to a rank-one strengthening with $|T|=1$, which leads to relaxations of $X_{2\times 2}$ similar to the perspective relaxation. Thus, we may argue that the perspective relaxation is required to describe $\clconv(X_{2\times 2})$. \qed
    \end{remark}

\revised{
\section{A Mixed-integer Linear Formulation for $P$}\label{sec:MILP}

The polyhedron $P$ can (in theory) be studied using standard methods from mixed-integer linear optimization. However, the vertex representation of $P$ is often not convenient, as most techniques require that the polyhedron be described explicitly via linear inequalities. Thus, in this section, we  present such a mixed-integer \textit{linear} formulation for the vertices of polytope $P$ when the Hessian matrix $Q$ is positive definite. 

First, we describe the linear equalities necessary for $P$. Throughout this section, for ease of exposition, for a given $S\subseteq[n]$, we permute the rows and columns of $Q$ such that indices in $S$ appear first. 

\begin{proposition}\label{lem:valideq}
For any $(z, W) \in P$,
\begin{equation}\label{eq:valideq}
    \sum_{k} Q_{ik} W_{ki} = z_i, \quad \forall i \in [n]. 
\end{equation}
\end{proposition}
\begin{proof}
For any $S \subseteq [n]$, $(\hat \ones_{S}, \hat Q_{S}^{-1}) \in P$, we have \begin{equation}\label{eq:mproduct}
\hat Q_{S}^{-1}Q = \begin{pmatrix}
Q_S^{-1} & 0 \\
0 & 0 
\end{pmatrix}
\begin{pmatrix}
Q_{S} & Q_{S, [n] \backslash S} \\
Q_{S, [n] \backslash S}^{\top} & Q_{[n] \backslash S}
\end{pmatrix}
=
\begin{pmatrix}
I_{|S|} & Q_{S}^{-1} Q_{S, [n] \backslash S} \\
0 & 0 
\end{pmatrix}.\end{equation}
Observe that the $i^{th}$ diagonal entry of $\hat Q_{S}^{-1} Q$ is one if $i \in S$ and zero otherwise. Since at all extreme points of $P$ we have $z=\hat \ones_S$ and $W=\hat Q_S^{-1}$ for some $S\subseteq [n]$, it follows that $(WQ)_{ii}=(\hat Q_{S}^{-1} Q)_{ii}=z_i$.
\end{proof}

Since $P$ satisfies $n$ linearly independent equalities, we immediately get insights into the dimension of $P$.
\begin{corollary}\label{cor:dimension}
The dimension of $P$ satisfies $\text{dim}(P)\leq n(n+1)/2$. If $Q$ is dense and $Z=\{0,1\}^n$, then this bound is tight.
\end{corollary}
\begin{proof}
Polyhedron $P$ has $n+n^2$ variables, but symmetry constraints $W_{ij}=W_{ji}$ and equalities \eqref{eq:valideq} imply the upper bound on the dimension. If $Q$ is dense, the set of points $(\hat\ones_{\{i,j\}},Q_{\{i,j\}}^{-1})_{i\neq j}$ are $n(n+1)/2$ affinely independent points of $P$, because each point is the unique one satisfying $W_{ij}\neq 0$. Together with point $(0,0)$, we find the required $n(n+1)/2+1$ affinely independent points in $P$.
\end{proof}

From Corollary~\ref{cor:dimension}, we see that (under mild conditions) there are no other equalities in the description of $P$. In order to construct a mixed-integer linear formulation for the vertices of $P$, we will use big-M constraints. Lemmas~\ref{lem:bigM} and \ref{lem:bigM2} are necessary to identify valid bounds for coefficients $M$.

\begin{lemma}\label{lem:bigM}
For any $S \subseteq [n]$, $Q^{-1} \succeq \hat Q_{S}^{-1}$ and $\|\hat Q_{S}^{-1}\|_{\infty} \leq \lambda_{\max}(Q^{-1})$.
\end{lemma}
\begin{proof}
To prove $Q^{-1} \succeq \hat Q_{S}^{-1}$ for $S \subseteq [n]$, it suffices to show $I \succeq Q^{1/2} \hat Q_{S}^{-1} Q^{1/2}$. Since switching the order of matrix multiplication does not change the set of nonzero eigenvalues, the nonzero eigenvalues of $Q^{1/2} \hat Q_{S}^{-1}Q^{1/2}$ coincide with those of $\hat Q_{S}^{-1}Q $. From \eqref{eq:mproduct} one sees that $\hat Q_{S}^{-1}Q =
\begin{pmatrix} \small
I_{|S|} & Q_{S}^{-1} Q_{S, [n] \backslash S} \\
0 & 0 
\end{pmatrix}$ is an upper triangular matrix, which has a maximum eigenvalue of one. Then we conclude that $I \succeq Q^{1/2} \hat Q_{S}^{-1} Q^{1/2}$ and thus $Q^{-1} \succeq \hat Q_{S}^{-1}$. 

For the second part, 
it follows that for $i \in [n]$, $(\hat Q_{S}^{-1})_{ii}\leq Q_{ii}^{-1} \leq \lambda_{\max} (Q^{-1})$. Since $\hat Q_{S}^{-1} \succeq 0$, for any $i,j \in [n]$, $(\hat Q_{S}^{-1})_{ij}^2\leq (\hat Q_{S}^{-1})_{ii}(\hat Q_{S}^{-1})_{jj}$. As $\lambda_{\max}(Q^{-1})$ gives a uniform bound on the diagonal elements of $\hat Q_{S}^{-1}$, $\lambda_{\max}(Q^{-1})$ also bounds the absolute value of the off-diagonal elements of $\hat Q_{S}^{-1}$. 
\end{proof}

Next, we define 
\begin{equation}\label{eq:defM}M \defeq \lambda_{\max}(Q^{-1}) \max_{i \in [n]} \big \{\|Q_{\{i\}}\|_{2} \big \}
\end{equation} 
and prove that $M$ provides a bound for the off-diagonal elements of $\hat Q_{S}^{-1} Q$ 
for any $S \subseteq [n]$ in the following lemma.

\begin{lemma}\label{lem:bigM2}
For any $S \subseteq [n]$, the off-diagonals of $\hat Q_{S}^{-1} Q$ are bounded by $M$.
\end{lemma}
\begin{proof}
 Note that $\hat Q_{S}^{-1} Q = \begin{pmatrix}
I_{|S|} & Q_{S}^{-1} Q_{S, [n]\backslash S} \\
0 & 0 \end{pmatrix}$. For any $j \notin S$, \begin{align*}\|Q_{S}^{-1} Q_{S, \{j\}}\|_{\infty} \leq \|Q_{S}^{-1} Q_{S, \{j\}}\|_{2} &\leq \lambda_{\max}(Q_{S}^{-1}) \|Q_{S, \{j\}}\|_{2} \\
&= \lambda_{\max} (\hat Q_{S}^{-1}) \|Q_{S, \{j\}}\|_{2} \leq \lambda_{\max}(Q^{-1}) \|Q_{\{j\}}\|_{2},\end{align*} where the last inequality follows from Lemma \ref{lem:bigM}. 
\end{proof}

One can make a few observations about $P = \{(\hat \ones_{S}, \hat Q_{S}^{-1})\}_{\hat \ones_{S} \in Z}$. Note that at extreme points of $P$, $W=\hat Q_S^{-1}$ for some $S$. Thus, for any extreme point $(z, W) \in P$, $W_{ij}$ is nonzero only if $z_i = z_j = 1$. Moreover, for any $S \subseteq [n]$, $(\hat \ones_{S}, \hat Q_{S}^{-1}) \in P$, $Q \hat Q_{S}^{-1} =  QW=\begin{pmatrix}
I_{|S|} & 0  \\
Q_{S, [n] \backslash S}^{\top} Q_{S}^{-1} & 0 
\end{pmatrix}$, and the off-diagonal entries in the $i^{th}$ row of $Q W$ are all zeros if $i \in S$. These two observations lead to the formulation in the following proposition.
 \begin{proposition}\label{prop:bilinear} 
 The extreme points of $P$  are described as 
\begin{align*}
\left\{(\hat e_{S}, \hat Q_{S}^{-1})_{\hat e_S \in Z}  \right\} = \Big\{(z, W) &\in Z \times \R^{n \times n} \; | \; \sum_{k = 1}^{n} Q_{ik} W_{ki} = z_i, \; \forall i \in [n], \\
&-M(1 - z_i) \leq \sum_{k = 1}^{n} Q_{ik} W_{kj} \leq M(1 - z_i), \; \forall i \neq j, \\
&|W_{ij}| \leq \lambda_{\max}(Q^{-1}) \min\{z_i, z_j\}, \; \forall i,j \in [n]\Big\}.
\end{align*}

\end{proposition}
\begin{proof}
For any $z = \hat e_{S} \in Z$, the constraint
\begin{displaymath}\label{bi:bigM}
|W_{ij}| \leq \lambda_{\max}(Q^{-1}) \min \{z_i, z_j\}, \quad \forall i,j \in [n],
\end{displaymath}
implies that $W_{ij} = 0$ if either $i$ or $j$ is not in $S$. For $i \in S$, we have 
\begin{align}
& \sum_{k = 1}^{n} Q_{ik} W_{ki} = 1 \label{bi:cons1} \\
& \sum_{k = 1}^{n} Q_{ik} W_{kj} = 0, \quad \forall j \neq i\label{bi:cons2}.
\end{align}
Inequalities \eqref{bi:cons1} and \eqref{bi:cons2} imply that 
$\begin{pmatrix} Q_{S} & Q_{S, [n] \backslash S}\end{pmatrix} \begin{pmatrix} W_{S} \\ W_{S, [n] \backslash S}^{\top}\end{pmatrix} = I$. Since $W_{S, [n] \backslash S} = 0$, we have $Q_S W_{S} = I$ and $W = \hat Q_{S}^{-1}$.  
Therefore, $Q \hat Q_{S}^{-1} = \begin{pmatrix}
I & 0 \\
Q_{S, [n]\backslash S}^{\top} Q_S^{-1} & 0
\end{pmatrix}$. It is clear that the off-diagonal elements in the $i^{th}$ row are all zero if $i \in S$, otherwise (if $i\not\in S$) they are bounded by $M$ according to Lemma ~\ref{lem:bigM2}. In other words,  constraints
\begin{align*}
& -M (1 - z_i) \leq \sum_{k = 1}^{n} Q_{ik} W_{kj} \leq M(1 - z_i), \quad \forall j \neq i
\end{align*}
hold. Moreover, thanks to Lemma ~\ref{lem:bigM}, the constraints 
\begin{equation}
|W_{ij}| \leq \lambda_{\max}(Q^{-1}) \min\{z_i, z_j\}, \; \forall i,j \in [n]
\end{equation}
hold at $W = \hat Q_{S}^{-1}$ and $z = \hat \ones_{S}$ as well.
\end{proof}

	 Proposition~\ref{prop:bilinear} allows us to give a mixed-integer \textit{linear} formulation for the \MIQO\ problem \eqref{eq:miqo}. Substituting the mixed-integer linear representation of $P$ given in Proposition ~\ref{prop:bilinear} in the equivalent \MIQO\ formulation \eqref{eq:miqoRelax3-1}, 
we arrive at an \textit{explicit} mixed-integer linear formulation for \eqref{eq:miqo}:
\begin{subequations}
\label{form:MILP}
\begin{align}
\min_{z, W} \; & -\frac{1}{2} a^{\top} W a + b^{\top} z  \\
\text{s.t.} \; & \sum_{k = 1}^{n} Q_{ik} W_{ki} = z_i, \quad \forall i \in [n]  \\
\text{(MILO)} \quad \quad \ \  \ & -M (1 - z_i) \leq \sum_{k = 1}^{n} Q_{ik} W_{kj} \leq M(1 - z_i), \quad \forall i \neq j \\
& |W_{ij}| \leq \lambda_{\max}(Q^{-1}) \min \{z_i, z_j\}, \quad \forall i,j \in [n]  \\
& z \in Z,
\end{align}
\end{subequations}
where $M$ is defined in \eqref{eq:defM}.

We point out that the mixed-integer representation of $P$ in Proposition~\ref{prop:bilinear} relies on big-M constraints and, therefore, it is not a strong formulation. Nonetheless, advanced mixed-integer linear optimization solvers have a plethora of built-in techniques to improve such formulations.  Preliminary computations using Gurobi indicate the following findings:
\begin{enumerate}
\item The natural relaxation of \eqref{form:MILP} is very weak and, therefore, \eqref{form:MILP} results in worse performance than alternative (nonlinear) formulations for problem~\eqref{eq:miqo} in most cases.
\item In some cases, however, and notably when the matrix $Q$ is sparse, Gurobi improves the relaxation in presolve to the point where the problems are solved at the root node, faster than existing formulations for \eqref{eq:miqo}. This situation illustrates that (in some cases), due to the polyhedrality of $P$, existing methods can improve even weak relaxations, whereas similar improvements are not available for nonlinear formulations.
\end{enumerate}

Detailed computational results are presented in Appendix~\ref{sec:numerical}. Overall, the results illustrate the potential benefits of reducing convexification to describing a polyhedral set, but also indicate that much work remains to be done for deriving better relaxations of $P$.

}

\section{Conclusion}\label{sec:conclusion}

In this paper, we \revised{first describe} the convex hull of the epigraph of a convex quadratic function with indicators \revised{in an extended space, which is given by one semi-definite constraint, and an exponential system of linear inequalities defining the convex hull of a polytope, $P$ (or $P_F$).  
We then derive the convex hull description in the original space as a semi-infinite conic quadratic program.   
Furthermore, we give a \emph{compact} mixed-integer linear representation of the vertices of the polytope $P$ that results in the first compact mixed-integer \textit{linear} formulation of \MIQO\ problems. While this is a weak formulation, our preliminary computational experience indicates that for a class of  sparse problems, 
off-the-shelf solvers are able to take advantage of the developments in MILO to improve the formulation substantially and it is competitive if not better than state-of-the-art approaches. To translate our theoretical developments into effective practical methods, it is crucial to exploit the structure of $P$. In our ongoing work, we explore the case when $Q$ is a Stieltjes matrix for which $P$ has a nice structure that allows us to use our results directly without resorting to the MILO formulation.   
Our results provide a unifying framework for several convex relaxations of \MIQO\ problems in the literature and can also be used to evaluate their strength.} 

\section*{Acknowledgments}
We thank the AE and three reviewers for their suggestions that improved the presentation. Alper Atamt\"urk is supported, in part, by 
NSF AI Institute for Advances in Optimization Award 211253, 
NSF grant 1807260, and DOD ONR grant 12951270.
Andr\'es G\'omez is supported, in part, by NSF grant 2006762 and AFOSR grant FA9550-22-1-0369. 
Simge K\"u\c{c}\"ukyavuz and Linchuan Wei are supported, in part, by  NSF grant 2007814 and DOD ONR grant N00014-19-1-2321.

\bibliographystyle{abbrvnat}
\bibliography{reference}

\begin{thebibliography}{35}
\providecommand{\natexlab}[1]{#1}
\providecommand{\url}[1]{\texttt{#1}}
\expandafter\ifx\csname urlstyle\endcsname\relax
  \providecommand{\doi}[1]{doi: #1}\else
  \providecommand{\doi}{doi: \begingroup \urlstyle{rm}\Url}\fi

\bibitem[POR()]{PORTA}
{PO}lyhedron {R}epresentation {T}ransformation {A}lgorithm.
\newblock \url{https://porta.zib.de/#download}.
\newblock Accessed: 2021-11-20.

\bibitem[Akt{\"u}rk et~al.(2009)Akt{\"u}rk, Atamt{\"u}rk, and
  G{\"u}rel]{akturk2009strong}
M.~S. Akt{\"u}rk, A.~Atamt{\"u}rk, and S.~G{\"u}rel.
\newblock A strong conic quadratic reformulation for machine-job assignment
  with controllable processing times.
\newblock \emph{Operations Research Letters}, 37:\penalty0 187--191, 2009.

\bibitem[Akt{\"u}rk et~al.(2010)Akt{\"u}rk, Atamt{\"u}rk, and
  G{\"u}rel]{aag:scheduling}
M.~S. Akt{\"u}rk, A.~Atamt{\"u}rk, and S.~G{\"u}rel.
\newblock Parallel machine match-up scheduling with manufacturing cost
  considerations.
\newblock \emph{Journal of Scheduling}, 13:\penalty0 95--110, 2010.

\bibitem[Albert(1969)]{albert1969conditions}
A.~Albert.
\newblock Conditions for positive and nonnegative definiteness in terms of
  pseudoinverses.
\newblock \emph{SIAM Journal on Applied Mathematics}, 17\penalty0 (2):\penalty0
  434--440, 1969.

\bibitem[Anstreicher and Burer(2021)]{anstreicher2021quadratic}
K.~M. Anstreicher and S.~Burer.
\newblock Quadratic optimization with switching variables: The convex hull for
  $n= 2$.
\newblock \emph{Mathematical Programming}, 188:\penalty0 421--441, 2021.

\bibitem[Atamt{\"u}rk and G{\'o}mez(2018)]{atamturk2018strong}
A.~Atamt{\"u}rk and A.~G{\'o}mez.
\newblock Strong formulations for quadratic optimization with {M}-matrices and
  indicator variables.
\newblock \emph{Mathematical Programming}, 170:\penalty0 141--176, 2018.

\bibitem[Atamt\"urk and G\'omez(2019)]{atamturk2019rank}
A.~Atamt\"urk and A.~G\'omez.
\newblock Rank-one convexification for sparse regression.
\newblock \emph{arXiv preprint arXiv:1901.10334}, 2019.

\bibitem[Atamt\"urk and G\'omez(2020)]{atamturk2020supermodularity}
A.~Atamt\"urk and A.~G\'omez.
\newblock Supermodularity and valid inequalities for quadratic optimization
  with indicators.
\newblock \emph{arXiv preprint arXiv:2012.14633}, 2020.

\bibitem[Atamt\"urk et~al.(2021)Atamt\"urk, G\'omez, and
  Han]{atamturk2021sparse}
A.~Atamt\"urk, A.~G\'omez, and S.~Han.
\newblock Sparse and smooth signal estimation: Convexification of
  {$\ell_0$}-formulations.
\newblock \emph{Journal of Machine Learning Research}, 22\penalty0
  (52):\penalty0 1--43, 2021.

\bibitem[Bach(2019)]{bach2019submodular}
F.~Bach.
\newblock Submodular functions: from discrete to continuous domains.
\newblock \emph{Mathematical Programming}, 175:\penalty0 419--459, 2019.

\bibitem[Bertsimas and King(2015)]{bertsimas2015or}
D.~Bertsimas and A.~King.
\newblock {OR} forum---an algorithmic approach to linear regression.
\newblock \emph{Operations Research}, 64:\penalty0 2--16, 2015.

\bibitem[Bertsimas et~al.(2021)Bertsimas, Cory-Wright, and
  Pauphilet]{bertsimas2020mixed}
D.~Bertsimas, R.~Cory-Wright, and J.~Pauphilet.
\newblock Mixed-projection conic optimization: A new paradigm for modeling rank
  constraints.
\newblock \emph{Operations Research}, 2021.
\newblock \doi{10.1287/opre.2021.2182}.
\newblock Article in Advance.

\bibitem[Bien et~al.(2013)Bien, Taylor, and Tibshirani]{bien2013lasso}
J.~Bien, J.~Taylor, and R.~Tibshirani.
\newblock A lasso for hierarchical interactions.
\newblock \emph{Annals of Statistics}, 41\penalty0 (3):\penalty0 1111, 2013.

\bibitem[Bienstock(1996)]{B:miqp}
D.~Bienstock.
\newblock Computational study of a family of mixed-integer quadratic
  programming problems.
\newblock \emph{Mathematical Programming}, 74\penalty0 (2):\penalty0 121--140,
  1996.

\bibitem[Ceria and Soares(1999)]{CS:dis-conv}
S.~Ceria and J.~Soares.
\newblock Convex programming for disjunctive convex optimization.
\newblock \emph{Mathematical Programming}, 86:\penalty0 595--614, 1999.

\bibitem[Chen et~al.(2014)Chen, Ge, Wang, and Ye]{chen2014complexity}
X.~Chen, D.~Ge, Z.~Wang, and Y.~Ye.
\newblock Complexity of unconstrained $l_2-l_p$ minimization.
\newblock \emph{Mathematical Programming}, 143\penalty0 (1):\penalty0 371--383,
  2014.

\bibitem[Cozad et~al.(2014)Cozad, Sahinidis, and Miller]{cozad2014learning}
A.~Cozad, N.~V. Sahinidis, and D.~C. Miller.
\newblock Learning surrogate models for simulation-based optimization.
\newblock \emph{AIChE Journal}, 60\penalty0 (6):\penalty0 2211--2227, 2014.

\bibitem[Dong et~al.(2015)Dong, Chen, and Linderoth]{dong2015regularization}
H.~Dong, K.~Chen, and J.~Linderoth.
\newblock Regularization vs. relaxation: A conic optimization perspective of
  statistical variable selection.
\newblock \emph{arXiv preprint arXiv:1510.06083}, 2015.

\bibitem[Frangioni and Gentile(2006)]{Frangioni2006}
A.~Frangioni and C.~Gentile.
\newblock Perspective cuts for a class of convex 0--1 mixed integer programs.
\newblock \emph{Mathematical Programming}, 106:\penalty0 225--236, 2006.

\bibitem[Frangioni and Gentile(2007)]{Frangioni2007}
A.~Frangioni and C.~Gentile.
\newblock {SDP} diagonalizations and perspective cuts for a class of
  nonseparable {MIQP}.
\newblock \emph{Operations Research Letters}, 35:\penalty0 181--185, 2007.

\bibitem[Frangioni et~al.(2020)Frangioni, Gentile, and
  Hungerford]{frangioni2018decompositions}
A.~Frangioni, C.~Gentile, and J.~Hungerford.
\newblock Decompositions of semidefinite matrices and the perspective
  reformulation of nonseparable quadratic programs.
\newblock \emph{Mathematics of Operations Research}, 45\penalty0 (1):\penalty0
  15--33, 2020.

\bibitem[Gao and Li(2011)]{Gao2011}
J.~Gao and D.~Li.
\newblock Cardinality constrained linear-quadratic optimal control.
\newblock \emph{IEEE Transactions on Automatic Control}, 56:\penalty0
  1936--1941, 2011.

\bibitem[G{\"u}nl{\"u}k and Linderoth(2010)]{Gunluk2010}
O.~G{\"u}nl{\"u}k and J.~Linderoth.
\newblock Perspective reformulations of mixed integer nonlinear programs with
  indicator variables.
\newblock \emph{Mathematical Programming}, 124:\penalty0 183--205, 2010.

\bibitem[Han and G{\'o}mez(2021)]{han2021compact}
S.~Han and A.~G{\'o}mez.
\newblock Compact extended formulations for low-rank functions with indicator
  variables.
\newblock \emph{arXiv preprint arXiv:2110.14884}, 2021.

\bibitem[Han et~al.(2020)Han, G{\'o}mez, and Atamt{\"u}rk]{hga:2x2}
S.~Han, A.~G{\'o}mez, and A.~Atamt{\"u}rk.
\newblock 2x2 convexifications for convex quadratic optimization with indicator
  variables.
\newblock \emph{arXiv preprint arXiv:2004.07448}, 2020.

\bibitem[He et~al.(2021)He, Han, G{\'o}mez, Cui, and Pang]{he2021comparing}
Z.~He, S.~Han, A.~G{\'o}mez, Y.~Cui, and J.-S. Pang.
\newblock Comparing solution paths of sparse quadratic minimization with a
  stieltjes matrix.
\newblock \emph{Department of Industrial and Systems Engineering, University of
  Southern California}, 2021.

\bibitem[Hochbaum(2001)]{hochbaum2001efficient}
D.~S. Hochbaum.
\newblock An efficient algorithm for image segmentation, {Markov} random fields
  and related problems.
\newblock \emph{Journal of the ACM (JACM)}, 48\penalty0 (4):\penalty0 686--701,
  2001.

\bibitem[Jeon et~al.(2017)Jeon, Linderoth, and Miller]{Jeon2017}
H.~Jeon, J.~Linderoth, and A.~Miller.
\newblock Quadratic cone cutting surfaces for quadratic programs with on--off
  constraints.
\newblock \emph{Discrete Optimization}, 24:\penalty0 32--50, 2017.

\bibitem[K{\"u}{\c{c}}{\"u}kyavuz et~al.(2020)K{\"u}{\c{c}}{\"u}kyavuz,
  Shojaie, Manzour, and Wei]{KSMW20}
S.~K{\"u}{\c{c}}{\"u}kyavuz, A.~Shojaie, H.~Manzour, and L.~Wei.
\newblock Consistent second-order conic integer programming for learning
  {Bayesian} networks.
\newblock \emph{arXiv preprint arXiv:2005.14346}, 2020.

\bibitem[Liu et~al.(2022)Liu, Fattahi, G{\'o}mez, and
  K{\"u}{\c{c}}{\"u}kyavuz]{liu2021graph}
P.~Liu, S.~Fattahi, A.~G{\'o}mez, and S.~K{\"u}{\c{c}}{\"u}kyavuz.
\newblock A graph-based decomposition method for convex quadratic optimization
  with indicators.
\newblock \emph{Mathematical Programming}, 2022.
\newblock \doi{10.1007/s10107-022-01845-0}.
\newblock Article in Advance.

\bibitem[Manzour et~al.(2021)Manzour, K{\"u}{\c{c}}{\"u}kyavuz, Wu, and
  Shojaie]{MKS21}
H.~Manzour, S.~K{\"u}{\c{c}}{\"u}kyavuz, H.-H. Wu, and A.~Shojaie.
\newblock Integer programming for learning directed acyclic graphs from
  continuous data.
\newblock \emph{{INFORMS} Journal on Optimization}, 3\penalty0 (1):\penalty0
  46--73, 2021.

\bibitem[Penrose(1955)]{penrose1955generalized}
R.~Penrose.
\newblock A generalized inverse for matrices.
\newblock In \emph{Mathematical Proceedings of the Cambridge Philosophical
  Society}, volume~51, pages 406--413. Cambridge University Press, 1955.

\bibitem[Wei et~al.(2020)Wei, G{\'o}mez, and
  K{\"u}{\c{c}}{\"u}kyavuz]{wei2020convexification}
L.~Wei, A.~G{\'o}mez, and S.~K{\"u}{\c{c}}{\"u}kyavuz.
\newblock On the convexification of constrained quadratic optimization problems
  with indicator variables.
\newblock In \emph{International Conference on Integer Programming and
  Combinatorial Optimization}, pages 433--447. Springer, 2020.

\bibitem[Wei et~al.(2022)Wei, G\'omez, and K\"u\c{c}\"ukyavuz]{wei2021ideal}
L.~Wei, A.~G\'omez, and S.~K\"u\c{c}\"ukyavuz.
\newblock Ideal formulations for constrained convex optimization problems with
  indicator variables.
\newblock \emph{Mathematical Programming}, 192\penalty0 ((1-2)):\penalty0
  57–--88, 2022.

\bibitem[Zheng et~al.(2014)Zheng, Sun, and Li]{zheng2014improving}
X.~Zheng, X.~Sun, and D.~Li.
\newblock Improving the performance of {MIQP} solvers for quadratic programs
  with cardinality and minimum threshold constraints: A semidefinite program
  approach.
\newblock \emph{INFORMS Journal on Computing}, 26:\penalty0 690--703, 2014.

\end{thebibliography}

\appendix

\section{Validity of inequalities \eqref{eq:original2x2}}\label{app:validity}

Here we directly check the validity of the inequalities in Example~\ref{ex:2x2}, which are repeated for convenience.
\begin{align*}
t\geq\max_{y\in \R_+^6}\;&\frac{y_1x_1^2+y_2x_2^2+(-y_1/d_1-y_2/d_2-y_3-y_4+y_5+y_6)x_1x_2}{(1/\Delta)y_4+(y_1/d_1-y_4/\Delta+y_5/\Delta)z_1+(y_2/d_2-y_4/\Delta+y_6/\Delta)z_2}\\
\text{s.t.}\;&4y_1y_2\geq (-y_1/d_1-y_2/d_2-y_3-y_4+y_5+y_6)^2,\; y_1+y_2\leq 1.
\end{align*}
If $z_1=z_2=x_1=x_2=0$, then the inequality reduces to $t\geq 0$. 
	If $z_1=1$ and $z_2=x_2=0$, the inequality reduces to 
	$$t\geq \max_{y\in \R_+^2}\frac{y_1x_1^2}{y_1/d_1+y_5/\Delta} \cdot$$
	The inequality can be maximized by setting $y_6=y_1/d_1$ and $y_2=y_3=y_4=y_5=0$, and reduces to $t\geq d_1x_1^2$.
	The case $z_2=1$, $z_1=x_1=0$ is identical.
	
	Finally, if $z_1=z_2=1$, then the inequality reduces to 
	\begin{equation}\label{eq:valid12}t\geq \max_{y\in \R_+^6}\frac{y_1x_1^2+y_2x_2^2+(-y_1/d_1-y_2/d_2-y_3-y_4+y_5+y_6)x_1x_2}{y_1/d_1+y_2/d_2-y_4/\Delta+y_5/\Delta+y_6/\Delta} \cdot\end{equation}
	
	Note that we can assume, without loss of generality, that $y_3=0$ (otherwise, if $y_3>0$, one can increase $y_4$ and reduce $y_3$ to obtain a feasible solution with better objective value). Let $\bar y=y_4-y_5-y_6$.  With these simplifications, \eqref{eq:valid12} reduces to 
	\begin{subequations}\label{eq:simplified}
		\begin{align}
		t\geq \max \;&\frac{y_1x_1^2+y_2x_2^2+(-y_1/d_1-y_2/d_2-\bar y)x_1x_2}{y_1/d_1+y_2/d_2-\bar y/\Delta}\\
		\text{s.t.}\;&4y_1y_2\geq (-y_1/d_1-y_2/d_2-\bar y)^2,\; y_1+y_2\leq 1\\
		&y_1,\;y_2\geq 0,\; \bar y \text{ free}.
		\end{align}
	\end{subequations}
	By taking the derivative of the objective with respect to $\bar y$, we conclude that (for fixed values of $y_1$ and $y_2$) the objective is monotone, and thus $\bar y$ may be assumed to be set at a bound. In particular, the rotated cone constraint holds at equality, and $\bar y=-y_1/d_1-y_2/d_2\pm 2\sqrt{y_1y_2}$. Thus, problem \eqref{eq:simplified} further reduces to 
	\begin{subequations}\label{eq:simplified2}
		\begin{align}
		t\geq \Delta\max \;&\frac{y_1x_1^2+y_2x_2^2\pm 2\sqrt{y_1y_2}x_1x_2}{y_1d_2+y_2d_1\pm 2\sqrt{y_1y_2}}\\
		\text{s.t.}\;&y_1+y_2\leq 1\label{eq:simplified2_scaling}\\
		&y_1,\;y_2\geq 0
		\end{align}
	\end{subequations}
	
Substitute $\bar y_1=\pm\sqrt{y_1}$ and $\bar y_2=\pm \sqrt{y_2}$. By multiplying by $(\bar y_1^2d_2+\bar y_2^2d_1+ 2\bar y_1\bar y_2)/(t\Delta)\geq 0$ on both sides of the inequality, we find that \eqref{eq:simplified2} is satisfied if and only if
$\forall \bar y_1,\bar y_2 \in \R$ satisfying $\bar y_1+\bar y_2\leq 1$, it holds
	$$\left\langle \begin{pmatrix}d_2/\Delta- x_1^2/t&1/\Delta- x_1x_2/t\\
	1/\Delta- x_1x_2/t&d_1/\Delta- x_2^2/t\end{pmatrix},\begin{pmatrix}\bar y_1^2&\bar y_1\bar y_2\\
	\bar y_1\bar y_2&\bar y_2^2\end{pmatrix}\right\rangle\geq 0,  $$
	which in turn holds if and only if
	\begin{align*}\begin{pmatrix}d_2/\Delta-x_1^2/t&1/\Delta-x_1x_2/t\\
	1/\Delta- x_1x_2/t&d_1/\Delta- x_2^2/t\end{pmatrix}\succeq 0 &\Longleftrightarrow \begin{pmatrix}t &x_1&x_2\\ x_1 & d_2/\Delta& 1/\Delta\\x_2&1/\Delta&d_1/\Delta\end{pmatrix}\succeq 0\\
	&\Longleftrightarrow t\geq d_1x_1^2-2x_1x_2+d_2x_2^2.\end{align*}

\revised{
	\section{Numerical Experiments}\label{sec:numerical}

Formulation MILO provides one way of utilizing Theorem ~\ref{thm:canonical} for general problems for which an explicit linear description of $P$ is not available. In this section, we discuss the practical effectiveness of MILO to solve problem \eqref{eq:miqo}. First, in \S\ref{sec:real}, we test MILO on best subset selection problems \eqref{eqn:bestsubset}. As MILO is a weak formulation due to big-M constraints, it is often outperformed by alternative formulations to solve MIQO problems in the literature. Then, in \S\ref{sec:synthetic}, we test the formulations in a class of graphical models which result in MIQO problems where matrix $Q$ is sparse. It turns out advanced optimization solvers are able to substantially improve the relaxation, and MILO has better practical performance than the usual alternatives for this class of problems.

 We compare MILO with the following alternative formulations:

 \noindent
	\textbf{Natural}: 
 The natural reformulation, where we replace the nonconvex constraint $x_i (1 - z_i) = 0$ in \eqref{eq:miqo} with $ |x_i| \leq 5  \|x^{*}\|_{\infty} z_i$, where $x^{*}$ denotes the optimal solution of the problem without binary variables or cardinality constraints. Observe that $5 \|x^{*}\|_{\infty}$ is not guaranteed to be a valid bound on $|x_i|$, thus this formulation may produce suboptimal solutions for \eqref{eq:miqo}.

 \noindent
	\textbf{Perspective}: The perspective reformulation \cite{aag:scheduling, CS:dis-conv, Frangioni2006, Gunluk2010} where we extract a diagonal term \textbf{diag}$(\delta)$ from $Q$ and  add the perspective constraints $s_i z_i \geq x_i^2, \; \forall i \in [n]$. We choose $\delta$ to be the minimum eigenvalue of $Q$ in our experiments. 
	\begin{align*}
	\min_{z, x, s} \; &\quad  \frac{1}{2} x^{\top} \left( Q - \textbf{diag}(\delta) \right) x + a^{\top} x + \frac{1}{2} \sum_{i = 1}^{n} \delta s_i + b^{\top} z \\
	\text{s.t.} \; & \quad s_i z_i \geq x_i^2, \; \quad \forall i \in [n] \nonumber \\
	& \quad z \in Z.\nonumber 
	\end{align*}

In all experiments, $Z$ is defined by a cardinality constraint, i.e., $Z= \{z \in \{0,1\}^{n} \; | \; \sum_{i = 1}^{n} z_i \leq r \}$, where $r=kn$ for a given sparsity parameter $0<k\le 1$, and $b=0$. The mixed-integer optimization problems are solved by Gurobi 9.0 on a laptop with Intel(R) Core(TM) i7-8750H 2.20 GHz and 32 GB RAM. We set the time limit to 30 minutes, and we use the default settings of Gurobi parameters.

\subsection{Best subset selection}\label{sec:real}

In this section, we solve the best subset selection problem ~\eqref{eqn:bestsubset} with varying $k$ on the benchmark datasets in 
Table ~\ref{tab:datasets}, available from the UCI machine learning repository\footnote{\url{https://archive.ics.uci.edu/ml/datasets.php}}. The performance measures considered are solution time, the number of nodes explored, and the initial optimality gap of the continuous relaxation. We also record the optimality gaps attained at the root node after presolve (in parentheses). Denoting the optimal objective value of a continuous relaxation by $\textbf{LB}$ and the exact optimal value by $\textbf{OPT}$, the initial optimality gap is calculated as 
\% gap = 100 $\times \frac{\textbf{OPT} - \textbf{LB}}{\textbf{OPT}}$.
For instances that hit the time limit, we report the average end gap in parentheses.

\begin{table}[!th]
		\caption{Benchmark datasets.}
		\label{tab:datasets}
		\vskip 1mm
		\centering
		{
			\begin{tabular}{c|cc}
			\hline
			\hline
			dataset & $m$ & $n$ \\
			\hline
			Housing & 506 & 13 \\
			Diabetes & 442 & 11 \\
			Servo & 167 & 19 \\
			AutoMPG & 392 & 25 \\ \hline \hline
			\end{tabular}
		}
	\end{table}

Table  \ref{tab:real}  shows the performance of the different formulations on these benchmark datasets. We observe that the relaxation quality of MILO is poor, with optimality gaps well above 100\% (in the range of $10^3-10^7$\%). Indeed, even though the objective of \eqref{eqn:bestsubset} has a trivial lower bound of $0$, the objective values produced by the continuous relaxation of MILO are in all cases negative. The bad relaxation quality leads to large numbers of branch-and-bound nodes and solution times. However, for the special case of $k=0.1$ on the first three datasets, 
Gurobi is able to close \emph{almost all} the gap at the root node and solve the problems with little or no branching. Thus,
while the  results clearly indicate that at the moment---in the context of a \emph{general} MIQO---standard methods\footnote{As expected, the perspective reformulation leads to better performance than the natural formulation in Housing. In the other datasets, the minimum eigenvalue of the matrix is close to $0$, and the perspective reformulation is not effective.} are better than the MILO formulation, in some cases, solvers might be able to exploit the polyhedrality of MILO. In the next section, we present experiments showcasing this phenomenon.

\subsection{Inference with graphical models}\label{sec:synthetic}

    Given a graph $G  = (V, E)$, we consider the following \MIQO\ problem
	\begin{subequations} \label{opt:sparse}
		\begin{align}
		\min_{z, x} \; & \quad \sum_{i \in V} \frac{1}{\sigma^2} (y_i - x_i)^2 + \sum_{(i,j) \in E} (x_i - x_j)^2 \\
		\text{s.t.} & \quad x_i (1 - z_i) = 0  \quad \forall i \in [n] \\
		&\quad \sum_{i = 1}^{n} z_i \leq k |V| .  
		\end{align}
	\end{subequations}
Problem \eqref{opt:sparse} arises in the sparse inference problem of a two-dimensional Gaussian Markov random field (GMRF), see \cite{liu2021graph} for an in-depth discussion. 
	
	The graph $G$ we consider in our experiment is a two-dimensional $10 \times 10$ grid. 	The corresponding Hessian matrix $Q$ in problem \eqref{opt:sparse} is sparse: each row has at most five nonzero entries (including the diagonal element). We use the random instances from \cite{he2021comparing}, available at \url{https://sites.google.com/usc.edu/gomez/data}, where $y_i = x_i + \mathcal N(0, \sigma)$ is a noisy observation of $x$, and there are three randomly sampled $3 \times 3$ blocks of $x$ to be nonzero. Note that $\sigma$ affects both the noise level of $y$ and the diagonal dominance of $Q$ in \eqref{opt:sparse}, with small noise values $\sigma$ resulting in problems with larger diagonal dominance.
	We test on $\sigma = 0.1, 0.2, 0.3, 0.4, 0.5$ and sparsity levels $k = 0.1, 0.2, 0.3, 0.4, 0.5$. For each $\sigma$, we use five randomly generated instances and report the average statistics.

	\begin{landscape}
	\begin{table}[!th]
	\caption{Performance of MILO, Natural and Perspective on the datasets in Table \eqref{tab:datasets}.}
	\label{tab:real}
	\vskip 1mm
	\centering
	 \small{

\begin{tabular}{ccccccccccccccc} \hline \hline
dataset & & $k$ & & \multicolumn{3}{l}{MILO} & & \multicolumn{3}{l}{Natural} & & \multicolumn{3}{l}{Perspective} \\\cline{5-7}\cline{9-11}\cline{13-15}
& & & & \% gap &\#node & time(\%endgap) & & \% gap & \#node & time  & & \% gap & \#node & time  \\
\hline
        \multirow{5}{*}{Housing} & & 0.1 & & 2.15E3(0) & 1 & 0.02 & & 43 & 1 & 0.041 & & 26 & 23 & 0.057 \\ 
        & & 0.2 & &
        5.55E3(4.0E3) & 91 & 0.16 & & 28.2 & 53 & 0.037 & & 16 & 77 & 0.137  \\ 
        & & 0.3 & & 
        9.01E3(6.8E3) & 302 & 0.35 & & 19.3 & 19 & 0.04 & & 11 & 105 & 0.146\\
        & & 0.4 & & 1.26E4(1.20E4) & 1151 & 0.51 & & 11.1 & 20 & 0.041 & & 5 & 65 & 0.117\\ 
        & & 0.5 & &
        1.43E4(1.39E4) & 1667 & 0.58 & & 8.7 & 18 & 0.043 & & 5 & 105 & 0.156\\ 
        \hline
       \multirow{5}{*}{Diabetes} & & 0.1 & & 
        4.24E3(0) & 1 & 0.02 & & 26.5 & 1 & 0.04 & & 25 & 21 & 0.136\\ 
        & & 0.2 & &
        1.03E4(9.3E3) & 66 & 0.29 & & 10.8 & 7 & 0.036 & & 10 & 101 & 0.294\\ 
        & & 0.3 & & 1.61E4(1.46E4) & 222 & 0.31 & & 7.2 & 20 & 0.042 & & 7 & 251 & 0.585 \\ 
        & & 0.4 & &
        2.17E4(1.46E4) & 424 & 0.27 & & 5.1 & 25 & 0.049 & & 5 & 386 & 0.724\\ 
        & & 0.5 & &
        2.37E4(2.33E4) & 662 & 0.37 & & 1.8 & 37 & 0.053 & & 2 & 352 & 0.736\\ 
        \hline
        \multirow{5}{*}{Servo} & & 0.1 & & 7.32E4(0) & 0 & 0.05 & & 40.8 & 1 & 0.034 & & 41 & 37 & 0.349\\
        & & 0.2 & &
        1.91E6(1.51E4) & 1541 & 1.1 & & 22.5 & 158 & 0.063 & & 22 & 1760 & 5.733\\ 
        & & 0.3 & &
        3E6(2.72E6) & 17556 & 8.88 & & 15 & 1107 & 0.112 & & 15 & 2534 & 4.075\\ 
        & & 0.4 & &
        3.88E6(3.71E6) & 40491 & 47.27 & & 8 & 2536 & 0.172 & & 8 & 4333 & 4.027\\
        & & 0.5 & &
        4.43E6(4.37E6) & 1.2E5 & 144.74 & & 1.8 & 2103 & 0.149 & & 2 & 7241 & 4.597\\ 
        \hline
        \multirow{5}{*}{AutoMPG} & & 0.1 & &
        7.7E6(6.16E6) & 549 & 2.71 & & 51.3 & 177 & 0.062 & & 51 & 549 & 4.315\\ 
        & & 0.2 & & 
        2.26E7(2.26E7) & 18932 & 72.13 & & 28.6 & 1320 & 0.156 & & 29 & 2015 & 4.385\\ 
        & & 0.3 & &
        3.51E7(3.38E7) & 3.4E5 & 844.37 & & 20.3 & 1.27E4 & 0.469 & & 20 & 3540 & 5.681\\ 
        & & 0.4 & &
        4.39E7(4.39E7) & 4.2E6 & 1800(3.28E5) & & 9.2 & 1.01E5 & 0.483 & & 9 & 1.59E4 & 12.64\\ 
        & & 0.5 & &
        5.76E7(5.71E7) & 9.46E6 & 1800(7.36E5) & & 3.9 & 9669 & 0.343 & & 4 & 1.04E4 & 7.508\\ \hline    \hline
        \multicolumn{15}{l}{1.\ A 1800 second solution time means Gurobi hits the time limit, and we report the best optimality gap in the following parenthesis.} \\
        \multicolumn{15}{l}{2.\ For MILO, the gap after Gurobi's presolve is reported in the following parenthesis.}
\end{tabular}
}
\end{table}

\end{landscape}

Table~\ref{tab:synthetic} summarizes the results. Similar to the experiments reported in \S\ref{sec:real}, the continuous relaxation of MILO is the worst among the three formulations, with gaps well over 100\%. However, in this case, Gurobi closes virtually all optimality gap in \emph{all} the instances, and the problems are solved very fast with at most one branch-and-bound node. The overall performance is significantly better than using the natural MIQO formulation, and also better than the perspective reformulation for these instances.

We conjecture that sparsity of $Q$, which leads to sparsity in the linear constraints of the MILO formulation, allows Gurobi to perform significant bound tightening in presolve. In contrast, Gurobi is unable to achieve a similar improvement with a nonlinear formulation. This clearly showcases the benefit of reducing the convexification of $X$ to describing a polyhedron in such cases.

	\begin{landscape}
	\begin{table}[!th]
	    \centering
	    \vskip 1mm
	    \caption{Performance of MILO, Natural, and Perspective formulations on graphical models.} \label{tab:synthetic}
	    
\small{
\begin{tabular}{ccccccccccccccc}
\hline \hline
$\sigma$ & & $k$ & & \multicolumn{3}{l}{MILO} & & \multicolumn{3}{l}{Natural} & & \multicolumn{3}{l}{Perspective} \\\cline{5-7}\cline{9-11}\cline{13-15}
& & & & \% gap &\#node & time & & \% gap & \#node & time(\%endgap)  & & \% gap & \#node & time  \\
\hline
\multirow{5}{*}{$0.1$} & & 0.1 & & 1598.83(0)   & 0   & 1.32 & & 21.5  & 77.8      & 0.36      &  & 0.8  & 6.2      & 0.17     \\
& & 0.2 & &
2257.57(0)   & 0.2 & 1.44 & & 2.5   & 995.2     & 0.37       & & 0.1  & 22.6     & 0.25     \\
& & 0.3 & &
2404.38(0)   & 0.8 & 1.66 & & 0.67  & 7.21E4  & 2.19       & & 0.1  & 22       & 0.27     \\
& & 0.4 & &
2473.30(0)   & 0.4 & 1.61 & & 0.25  & 1.83E7  & 252.27    & & 0    & 139      & 1.04     \\
& & 0.5 & &
2500.19(0)   & 0.8 & 1.77 & & 0.16  & 1.5E7 & 490.67     & & 0    & 42.4     & 0.46     \\
\hline
\multirow{5}{*}{$0.2$} & & 0.1 & &
983.97(0)    & 0   & 1.93 & & 16.6  & 239       & 2.02       & & 0.9  & 11.4     & 0.45     \\
& & 0.2 & &
1512.18(0)   & 0.4 & 2.12 & & 4.6   & 1.34E6  & 46.87      & & 0.15 & 262.6    & 2.76     \\
& & 0.3 & &
1783.68(0)   & 0.8 & 2.53 & & 2.48  & $3.79E6^{2}$ & $1281.78^{2}(0.2)$ & & 0.15 & 2363.6   & 10.66    \\
& & 0.4 & &
1958.07(0)   & 1   & 2.53 & & 1.34  & $3.75E7^{3}$ & $1487.74^{3}(0.2)$  &  & 0.11 & 2.70E4 & 21.31    \\
& & 0.5 & &
2045.22(0)   & 0.8 & 2.51 & & 0.69  & $3.74E7^{3}$ & $1341.41^{3}(1)$ & & 0.06 & 3.13E4 & 21.11    \\
\hline
\multirow{5}{*}{$0.3$} & &  0.1 & &
704.64(0)    & 0.4 & 2.53 & & 18.98 & 1.72E4  & 13.58     & & 1.34 & 65       & 1.12     \\
& & 0.2 & &
1291.39(0)   & 0.4 & 2.57 & & 9.12  & $6.02E6^{1}$ & $440.89^{1}(0.6)$    &  & 0.59 & 853.8    & 8.01     \\
& & 0.3 & &
1740.51(0)   & 1   & 3.8 & & 5.5   & $2.98E7^{4}$ & $1555.90^{4}(1.4)$ & & 0.58 & 5629.8   & 19       \\
& & 0.4 & &
2062.39(0)   & 1   & 3.16 & & 3.18  & $4.17E7^{5}$         & $1715.11^{5}(1.2)$          & & 0.33 & 3.40E4 & 34.34    \\
& & 0.5 & & 
2237.68(0)   & 1   & 3.72 & & 1.8   & $4.31E7^{3}$ & $1630.78^{3}(0.6)$ & & 0.27 & 4.48E4 & 38.31    \\
\hline
\multirow{5}{*}{$0.4$} & & 0.1 & &
547.32(0.01) & 0.4 & 2.64 & & 23.4  & 1.10E6  & 179.06     & & 1.86 & 527.2    & 6.42     \\
& & 0.2 & &
1188.23(0)   & 0.6 & 3.18 & & 14.41 & $1.84E7^{4}$  & $1414.86^{4}(2.6)$ & & 1.45 & 3.56E3 & 21.13    \\
& & 0.3 & &
1766(0)      & 1   & 4.97 & & 9.3   & $2.78E7^{5}$         & $1800.36^{5}(3.8)$          & & 1.24 & 3.33E4 & 35.67    \\
& & 0.4 & &
2215.48(0)   & 0.8 & 7.24 & & 5.77  & $3.46E7^{5}$         & $1800.26^{5}(2.6)$          & & 0.84 & 5.55E4 & 48.21    \\
& & 0.5 & &
2492.49(0)   & 1   & 8.9  & & 3.4   & $3.90E7^{5}$         & $1800.58^{5}(1.2)$          & & 0.55 & 4.98E4 & 48.94    \\
\hline
\multirow{5}{*}{$0.5$} & & 0.1 & &
483.62(0)    & 0.2 & 2.42 & & 26.27 & $3.94E6^{1}$  & $662.67^{1}(1.2)$ & & 2.71 & 1.78E3 & 18.24    \\
& & 0.2 & & 
1096.05(0)   & 1   & 5.86 & & 18.17 & $1.85E7^{5}$         & $1800.36^{5}(6.8)$          & & 2.72 & 2.36E4 & 42.9     \\
& & 0.3 & &
1728.40(0)   & 0.4 & 5.8 & & 12.07 & $2.50E7^{5}$         & $1800.41^{5}(5.4)$          & & 2.7  & 1.74E4 & 1.02E3 \\
& & 0.4 & &
2261.46(0)   & 0.8 & 8.17 & & 7.75  & $3.67E7^{5}$         & $1800.61^{5}(3.8)$          & & 1.55 & 4.10E5 & 191.38   \\
& & 0.5 & & 
2600.52(0)   & 0.4 & 9.55 & & 4.58  & $3.60E7^{5}$         & $1800.2^{5}(1.8)$          & & 0.93 & 4.32E5 & 160.87  \\
\hline \hline
\multicolumn{15}{l}{1.\ A super script $^{i}$ indicates that $i$ out of five instances hit the time limit.} \\
\multicolumn{15}{l}{2.\ For instances reaching the time limit, the average best optimality gap is reported in the parenthesis following the solution time.} \\
\multicolumn{15}{l}{3.\ For MILO, the gap after Gurobi's presolve is recorded in parentheses.}
\end{tabular}}
	\end{table}
	\end{landscape}

}

\end{document}